\documentclass[12pt]{article}
\usepackage{amsmath,amssymb,amsfonts,amsthm,fouridx}
\usepackage{color}
\usepackage{graphicx}
\usepackage{subfigure}
\usepackage{setspace}

\newcommand{\R}{{\mathbb R}}

\newcommand{\mop}{\mathop}

\newtheorem{ass}{Assumption}
\newtheorem{lemma}{Lemma}		
\newtheorem{defi}{Definition}   
\newtheorem{rema}{Remark}
\newtheorem{coro}{Corollary}
\newtheorem{thm}{Theorem}

\makeatletter
\newsavebox\myboxA
\newsavebox\myboxB
\newlength\mylenA

\newcommand*\xoverline[2][0.75]{%
    \sbox{\myboxA}{$\m@th#2$}%
    \setbox\myboxB\null
    \ht\myboxB=\ht\myboxA%
    \dp\myboxB=\dp\myboxA%
    \wd\myboxB=#1\wd\myboxA
    \sbox\myboxB{$\m@th\overline{\copy\myboxB}$}
    \setlength\mylenA{\the\wd\myboxA}
    \addtolength\mylenA{-\the\wd\myboxB}%
    \ifdim\wd\myboxB<\wd\myboxA%
       \rlap{\hskip 0.5\mylenA\usebox\myboxB}{\usebox\myboxA}%
    \else
        \hskip -0.5\mylenA\rlap{\usebox\myboxA}{\hskip 0.5\mylenA\usebox\myboxB}%
    \fi}
\makeatother
\date{}
\begin{document}


\title{Convergence and Synchronization in heterogeneous networks of smooth and piecewise smooth systems\thanks{Preprint submitted to Automatica on May 7th 2013. P. DeLellis, M. di Bernardo, D. Liuzza are with the Department of Systems and Computer Engineering. M. di Bernardo is also with the Department of Engineering Mathematics, University of Bristol, BS8 1UB, UK. Email:\{pietro.delellis\}\{mario.dibernardo\}\{davide.liuzza\}@unina.it}}
\author{Pietro DeLellis, Mario di Bernardo, Davide Liuzza}
\maketitle
\begin{abstract}                          
This paper presents a framework for the study of convergence when the nodes' dynamics may be both piecewise smooth and/or nonidentical across the network. Specifically, we derive sufficient conditions for global convergence of all node trajectories towards the same bounded region of their state space. The analysis is based on the use of set-valued Lyapunov functions and  bounds are derived on the minimum coupling strength required to make all nodes in the network converge towards each other. We also provide an estimate of the asymptotic bound $\epsilon$ on the mismatch between the node states at steady state. The analysis is performed both for linear and nonlinear coupling protocols. The theoretical analysis is extensively illustrated and validated via its application to a set of representative numerical examples.
\end{abstract}


\section{Introduction}  
The problem of taming the collective behaviour of a network of dynamical systems is one of the key challenges in modern control theory, see for example \cite{LiSl:11,Me:07} and references therein. Typically, the ``simplest'' problem is to make all agents in the network evolve asymptotically onto a common synchronous solution. This problem is relevant in a number of different applications \cite{BaTa:10,DoBu:12,EuGu:04,JeTo:00,KaBu:10} and has been the subject of much ongoing research (see for example \cite{cncs:07,ChBa:13,HiZh:12,YuChLu:09,YuDe:12}). 

In general, a network is modeled as an ensemble of $N$ interacting dynamical systems \cite{bola:06,neba:06} or ``agents''. Each system is described by a set of nonlinear ordinary differential equations (ODEs) of the form $\dot x_i = f_i(t,x_i)$, where $x_i \in \mathbb{R}^n$ is the state vector and $f_i: \mathbb{R}^+\times\mathbb{R}^n \rightarrow \mathbb{R}^n$ is a nonlinear vector field describing the system dynamics, often assumed sufficiently smooth and differentiable. The coupling between neighboring nodes is assumed to be a nonlinear function $\eta: \mathbb{R}^+\times\mathbb{R}^n \rightarrow \mathbb{R}^n$ (often called output or coupling function) of their states.
Hence, the equations of motion for the generic $i$-th system in the network are:
\begin{equation}\label{eq:general_net}
\frac{\mathrm{d}x_i}{\mathrm{d}t}=f_i(t,x_i)-c\sum_{j=1}^N { a_{ij} \eta(t,x_i, x_j)}, \qquad \forall i=1,\dots,N
\end{equation}
where $x_i$ represents the state vector of the $i$-th agent, $c$ is the overall strength of the coupling, and 
$a_{ij}=a_{ji}\ge0$ is positive if there is an edge between nodes $i$ and $j$ and $0$ otherwise.

Different strategies have been proposed to solve the problem of making all agents in the network converge onto the same solution. Examples include strategies for consensus in networks when the agents are linear and coupled diffusively \cite{OlFa:07,OlMu:04,YuCh:10}, adaptive approaches to consensus and synchronization \cite{ChBa:13,DaLe:12,dediga:09,dedigapo:09,LiZh:08,ZhKu:06} and methods based on distributed leader-follower (or pinning control) techniques among many others \cite{AbPo:12,ChCh:09,DediPo:11,LiCh:08,ZhLe:11}.

Most of the results available in the existing literature rely on the following assumptions which are essential to simplify the study of model \eqref{eq:general_net} and its convergence. Namely, it is often assumed that
\begin{enumerate}
\item the output functions are linear, time-invariant, and typically depending upon the mismatch between the states of neighbouring nodes, i.e.  $\eta(t, x_i,x_j)=\gamma_{ij}(x_i-x_j),\ \gamma_{ij}\in \mathbb{R}^+$;
\item the nodes'  vector fields, $f_i$, are sufficiently smooth and differentiable 
\item all nodes share the same dynamics, i.e. $f_i=f_j$, for $i,j=1,\ldots,N$.
\end{enumerate}
Under the above assumptions,
the stability and convergence of network \eqref{eq:general_net} have been investigated in depth over the last two decades, and interesting results have been obtained (for a review see \cite{bola:06,dedigo:10,neba:06,piro:01,xich:07}). 
For example, when all nodes are identical and described by smooth vector fields, conditions can be derived under which asymptotic convergence (or complete synchronization) is guaranteed. Namely, it is possible to prove that all nodes asymptotically converge onto the manifold in state space where $x_1=x_2=\ldots=x_N$.

Unfortunately, in many real-world networks it is often unrealistic to assume that all nodes share the same identical dynamics. Think for example of biochemical or power networks were parameter mismatches between agents are unavoidable and usually rather large  \cite{bale:00,DoBu:12,hich:06,kupa:04,thou:01,wabu:06}. The problem of coordination among heterogeneous nodes is relevant in Networked Cyber-Physical Systems \cite{KiSt:10,Le:08}. Also, in many cases the models in use to describe the dynamics of the nodes in the network are far from being continuous and differentiable. Notable cases include the coordinated motion of mechanical oscillators with friction \cite{he:02,peva:04,peva:03}, switching power devices \cite{mota:06,page:05}, switch-like models of behaviours of biological cells in pattern formation \cite{shsp:11}, and all those networks whose nodes' dynamics are affected by discontinuous events on a macroscopic timescale. 
The aim of this paper is to study the challenging open problem of characterising  convergence and synchronization in networks whose nodes' dynamics are nonidentical and possibly described by piecewise smooth vector fields.  

In this case, asymptotic convergence is only possible in specific cases; for example, when all nonidentical nodes share the same equilibrium \cite{xich:07}, for specific nodes' dynamics, or in the case where symmetries exist in the network structure \cite{duch:09,feko:05,zhhi:11}. Nonetheless, for more general complex network models, these assumptions have to be relaxed. Hence, when either a mismatch is present in the network parameters  and/or perturbations are added to the vector field of the nodes, it is often desirable to prove bounded (rather than asymptotic) convergence of all nodes towards each other.
As an example, in power networks, asymptotic convergence of all generator phases towards the same solution cannot be achieved and 
it is considered acceptable that the phase angle differences remain within given bounds \cite{DoBu:12,hich:06,kupa:04}.

In the literature, few results are available on bounded convergence of networks of nonidentical nodes. In particular, the case of parameters' mismatches is studied assuming that the nodes' dynamics are \emph{eventually dissipative} 
\cite{bebe:03}, or assuming a priori that the node trajectories are bounded \cite{huli:09,lilu:12}. Local stability of networked systems with small parameter mismatches is studied extending the Master Stability Function approach in \cite{subo:09}. As for additive perturbations, the
specific case of additive noise was considered in \cite{lich:07,phta:09}.
A first attempt on giving more general conditions for bounded convergence can be found in 
\cite{hizh:08}. However, the key assumptions guaranteeing global stability results are difficult to check in practice. Indeed, assumptions given in \cite{hizh:08} rely on boundedness of the average node vector field, defined as $\sum_{i=1}^N f_i/N$, and of its Jacobian evaluated on the average network trajectory, which is unknown a priori.

In networks of piecewise smooth systems, guaranteeing convergence is a cumbersome task even when the nodes' vector fields are identical and only few results are currently available \cite{da:02,rudi:11}. Specifically, in \cite{da:02}, local synchronization of two coupled continuously differentiable systems with a specific additive sliding action is guaranteed with conditions on the generalized Jacobian of the error system, while in \cite{rudi:11} convergence of a network of time switching systems is analyzed when the switching signal is synchronous between all nodes.

To the best of our knowledge, none of the approaches in the existing literature can deal with the generic case of networks characterized by the presence of both piecewise smooth and nonidentical nodes' dynamics. 
The main contributions of this paper can be summarized as follows.
\begin{enumerate}
\item Sufficient conditions are derived using  set-valued Lyapunov functions for global bounded convergence of all network nodes towards each other. Moreover, explicit bounds are estimated for the residual tracking error and the value of the minimum coupling strength among nodes guaranteeing convergence.
\item The classical assumption of linear diffusive coupling functions is relaxed. Our stability analysis also encompasses continuous or PWS nonlinear coupling protocols.
\item When applied to networks of nonidentical smooth systems, the general conditions derived in this paper give sufficient conditions for global bounded convergence that are much easier to check or verify if compared to those given in the existing literature reviewed above.
\item This paper significantly extends the preliminary results reported in \cite{da:02,rudi:11} guaranteeing boundedness of the synchronization error in networks of piecewise smooth systems. In particular, results of bounded convergence are found for a wider class of systems and of possible switching signals than those in \cite{da:02,rudi:11}.
\end{enumerate}
The rest of the paper is structured as follows. In Section \ref{sec:pws}, some background is given on PWS dynamical systems. Then, in Section \ref{sec:mathematicalpreliminaries}, the network model  of interest is presented together with relevant mathematical preliminaries used in the rest of the paper. In Section \ref{sec:totally},  bounded convergence for linearly coupled networks is investigated. The analysis is then extended  in Section \ref{sec:partially}  to the case of nonlinearly coupled networks. All through the presentation, a set of representative examples is used to illustrate the application of the theoretical derivation. Conclusions are drawn is Section \ref{sec:con}.

\section{Piecewise smooth dynamical systems}\label{sec:pws}
In this section, we introduce some notation and review some concepts and definitions on PWS dynamical systems that will be used throughout the paper.  $I_s$ denotes the $s\times s$ identity matrix, $1_s$ is the $s$-dimensional
 vector $[1,\ldots,1]^T$, $\left\|\cdot\right\|_p$ denotes the matrix (vector) $p$-norm, $\lambda_{\mathrm{max}}(M)$ denotes the maximum eigenvalue of a matrix $M$, and $\mathrm{diag}\{m_i\}_{i=1}^s$ is the $s\times s$ diagonal matrix whose diagonal elements are $m_1,\ldots,m_s$. Given a matrix $M$, its positive (semi) definiteness is denoted by $M>0$ ($M\ge 0$). Furthermore, with $\mathcal D$ we denote the set of diagonal matrices and with $\mathcal D^+$ the set of positive definite diagonal matrices.

Now, we give the definition of a PWS dynamical system according to \cite{dibu:07}, p.73.

\begin{defi}\label{def:pws}
Let us consider a finite collection of disjoint, open and non-empty sets  $\mathcal{S}_1,\dots \mathcal{S}_p$, such that $\mathrm{D}\subseteq\bigcup_{k=1}^p\bar{\mathcal{S}}_k\subseteq\R^{n}$ is a connected set, and that the intersection $\Sigma_{hk}:=\bar{\mathcal{S}}_h\cap\bar{\mathcal{S}}_k$ is either a $\R^{n-1}$ lower dimensional manifold or it is the empty set. 
A dynamical system $\dot{x}=f(t,x)$, with $f:\R^+\times\mathrm{D} \mapsto\R^n$, is called a piecewise smooth dynamical system when it is defined by a finite set of ODEs, that is, when
\begin{equation} 
f(t,x)=F_k(t,x)\qquad x\in\mathcal{S}_k,\,k=1,\ldots,p,\label{eq:pws_sys}
\end{equation} 
with each vector field $F_k(t,x)$ being smooth in both the state $x$ and the time $t$ for any $x\in\mathcal{S}_k$. Furthermore, each $F_k(t,x)$ is continuously extended on the boundary $\partial \mathcal{S}_k$. 
\end{defi}
Notice that in the above definition  the value the function $f(\cdot)$ assumes on the boundaries $\partial \mathcal{S}_k$ is left undefined. For PWS system \eqref{eq:pws_sys}, different solution concepts can be defined (see \cite{co:08} and references therein). In this paper, we focus on {\em Filippov solutions} \cite{fi:88}. These solutions are absolutely continuous curves $x(t):\R\mapsto\R^n$ satisfying, for almost all $t$, the differential inclusion:
\begin{equation}\label{eq:diff_inclusion}
\dot{x}(t)\in\mathcal{F}[f](t,x),
\end{equation}
where $\mathcal{F}[f](t,x)$ is the {\em Filippov set-valued function} $\mathcal{F}[f]:\R^+\times \R^n\mapsto\mathfrak{B}(\R^n)$,  with $\mathfrak{B}(\R^n)$  being the collection of all subsets in $\R^n$, defined as
\begin{equation}\label{eq:filippov_set_valued_map}
\mathcal{F}[f](t,x)=\bigcap_{\delta>0}\bigcap_{m(\mathcal{S})=0}\overline{co}\left\{f(t,\mathcal{B}_{\delta}(x)\textrm{\textbackslash} \mathcal{S}) \right\},
\end{equation}
$\mathcal{S}$ being any set of zero Lebesgue measure $m(\cdot)$, $\mathcal{B}_{\delta}(x)$  an open ball centered at $x$ with radius $\delta>0$, and $\overline{co}\left\{\mathcal{I}\right\}$ denoting the convex closure of a set $\mathcal{I}$.

We remark that, for the piecewise smooth system \eqref{eq:pws_sys}, a Filippov solution exists under the mild assumption of local essential boundedness of the vector field $f$, see \cite{co:08} for further details. In the rest of this paper, we assume that the PWS system \eqref{eq:pws_sys} is defined in the whole state space $\R^n$, so that $x\in\mathrm{D} \equiv \R^n$.

Computing the Filippov set-valued function (\ref{eq:filippov_set_valued_map}) can be a nontrivial task. Here, we report three useful rules that can be used to ease the computations \cite{PaSa:87}:

\begin{description}
\item[Consistency:]{\em  If $f:\R^+\times\R^n \mapsto\R^n$ is continuous at $(t,x)\in\R^+\times\R^n$, then  
\[
\mathcal{F}[f](t,x)=\left\{f(t,x)\right\}.
\]
\item[{\em Sum:}] If $f_1,f_2:\R^+\times\R^n\mapsto\R^n$ are locally bounded at $(t,x)\in\R^+\times\R^n$, then}
\[
\mathcal{F}[f_1+f_2](t,x)\subseteq\mathcal{F}[f_1](t,x)+\mathcal{F}[f_2](t,x).
\]
\end{description}
{\em Moreover, if either $f_1$ or $f_2$ is continuous at $(t,x)$, then the equality holds.}
\begin{description} 
\item[Product:]
{\em If $f_1,f_2:\R^+\times\R^n\mapsto\R^n$ are locally bounded at $(t,x)\in\R^+\times\R^n$, then}
\[
\mathcal{F}\left[\left(f_1^T,f_2^T\right)^T\right](t,x)\subseteq \mathcal{F}[f_1](t,x)\times\mathcal{F}[f_2](t,x).
\]
\end{description}
{\em Moreover, if either $f_1$ or $f_2$ is continuous at $(t,x)$, then equality holds.}

A PWS system is not differentiable everywhere in its domain. Nonetheless, as reported in \cite{cl:83}, the Rademacher's Theorem states that a function which is locally Lipschitz is differentiable almost everywhere (in the sense of Lebesgue). Then, it is useful to extend the classical gradient definition.
Denoting with $\Omega_u$ the zero-measure set of points at which a given function $u$ fails to be differentiable, we report the following definition \cite{cl:83,co:08}.

\begin{defi}
Let $u:\R^n\mapsto \R$ be a locally Lipschitz function, and let $\mathcal{S}\subset \R^n$ be an arbitrary set of zero measure, we define the {\em generalized gradient} (also termed {\em Clarke subdifferential}) $\partial u:\R^n\mapsto\mathfrak{B}(\R^n)$ of $u$ at any $x\in\R^n$ as
\[
\partial u(x)=co\left\{ \lim_{k\rightarrow \infty} \frac{\partial }{\partial x}u(x_k):x_k\rightarrow x,x_k\notin \mathcal{S}\cup \Omega_u \right\}.
\]
\end{defi}
Notice that, if $u$ is continuously differentiable, then it is possible to prove that $\partial u(x)=\left\{\frac{\partial }{\partial x} u(x)\right\}$, see \cite{co:08}.
\begin{defi}
\cite{co:08} Given a locally Lipschitz function $u:\R^n\mapsto \R$ and a vector field $f:\mathbb{R}^n\rightarrow\mathbb{R}^n$, the {\em set-valued Lie derivative} ${\mathop \mathcal{L} \limits_\sim}{}_{\mathcal{F}[f]}:\R^n\mapsto\mathfrak{B}(\R)$ of $u$ with respect to $\mathcal{F}[f]$ at $x$ is defined as
\[
{\mathop \mathcal{L} \limits_\sim}{}_{\mathcal{F}[f]}u(x):=\left\{a\in\R \,s.t. \textrm{ there exists $v\in\mathcal{F}[f](x) \Rightarrow \varrho^Tv=a$ for all $\varrho\in\partial u(x)$}\right\}.
\] 
\end{defi}
\begin{lemma} \cite{bace:99,co:08}
Let $x(t)$ be a solution of the differential inclusion \eqref{eq:diff_inclusion}, \eqref{eq:filippov_set_valued_map}, and let $u:\R^n\mapsto \R$ be locally Lipschitz and regular. Then, the following statements hold:
\begin{itemize} 
\item [i)] The composition $t\in\mathbb{R}^+\mapsto u(x(t))\in\R$ is differentiable for almost every $t$;
\item [ii)] The derivative of $t \mapsto u(x(t))$ satisfies
\[
\frac{d}{dt}u(x(t))\in{\mathop \mathcal{L} \limits_\sim}{}_{\mathcal{F}[f]}u(x)
\]
for almost every $t$.
\end{itemize}
\end{lemma}

Notice that a convex function is also regular, see, for instance, \cite{cl:83}.

To simplify the notation, in what follows the set valued function $\mathcal{F}[f](t,x)$ is equivalently denoted by $\mop{f}\limits_\sim(t,x)$, while an element of $\mop{f}\limits_\sim(t,x)$ is denoted by $\mop f\limits^{\sim}(t,x)$. Now, we define the class of {\em QUAD} PWS vector fields, that will be considered throughout the paper.
\begin{defi}\label{def:quad}
Similarly to what stated in \cite{dedi:11} we say that, given a pair of $n\times n$ matrices $P\in\mathcal D^+$,  $W\in\mathcal D$, a PWS vector field $f:\mathbb{R}^+ \times \R^n \rightarrow \mathbb{R}^n$ is QUAD(P,W) if and only if the following inequality holds:
\begin{equation}\label{eq:quad}
(x-y)^TP\left[\mathop f\limits ^{\sim}(t,x)-\mathop f\limits ^{\sim}(t,y)\right]\leq (x-y)^T W (x-y),
\end{equation}
for all $x,y\in\mathbb R^n,\, t\in\mathbb{R}^+, \mathop f\limits ^{\sim}(t,x)\in\mop{f}\limits_\sim (t,x), \mathop f\limits ^{\sim}(t,y)\in\mop{f}\limits_\sim (t,y)$.
\end{defi}
Note that this property is equivalent to the well-known one-sided Lipschitz condition for $P=I_n$ and $W=w I_n$ \cite{co:08}. Furthermore, the QUAD condition is also related to some relevant properties of the vector fields, such as contraction properties for smooth systems and the classical Lipschitz condition, see \cite{dedi:11} for further details.

We extend the QUAD condition to PWS systems as follows.
\begin{defi}\label{def:quada}
A PWS system  is said to be {\em QUAD(P,W) Affine} iff its vector field can be written in the form:
\begin{equation}\label{eq:quad_affine}
f(t,x(t))=h(t,x(t))+g(t,x(t)),
\end{equation}
where:
\begin{enumerate}
\item $h$ is either a continuous or piecewise smooth QUAD(P,W) function.
\item $g$ is either a continuous or piecewise smooth function such that there exists a positive scalar $M<+ \infty$ satisfying 
\begin{equation}
\left|\left|\mathop g\limits ^{\sim}(t,x(t))\right|\right|_2<M, \qquad\quad \forall x\in\mathbb{R}^{n}, \forall t\in\mathbb{R}^+,  \forall \mathop g\limits ^{\sim}(t,x(t))\in \mop g \limits _\sim (t,x(t)
\nonumber
\end{equation}
\end{enumerate}
\end{defi}

It is worth mentioning that QUAD Affine systems can exhibit sliding mode and chaotic solutions, so this hypothesis on the nodes' dynamics does not exclude typical behaviors that may arise in PWS systems (see Sec. \ref{sec:examples} for some representative examples). 


\section{Network model and problem statement}\label{sec:mathematicalpreliminaries}
In what follows, we analyze the general model \eqref{eq:general_net} of networks of nonidentical (piecewise) smooth systems, where we assume that the output function is either a nonlinear or linear function of the state mismatch among neighbouring nodes. Specifically, in the linear case we have:
\begin{equation}\label{eq:genericnetwork_linear}
\frac{\mathrm{d}x_i}{\mathrm{d}t}=f_i(t,x_i)-c\sum_{j=1}^N { a_{ij}\Gamma (x_i-x_j)},
\end{equation}
where $\Gamma \in\mathbb{R}^{n\times n}$ is the so-called inner coupling matrix determining what state variables are involved in the coupling (see \cite{bola:06,dedigo:10}). When the coupling is nonlinear we get instead:
\begin{equation}\label{eq:genericnetwork}
\frac{\mathrm{d}x_i}{\mathrm{d}t}=f_i(t,x_i)-c\sum_{j=1}^N { a_{ij}\eta(t,x_i-x_j)}.
\end{equation}

In the following sections, we investigate bounded convergence in the networks above. To give a formal definition of bounded convergence, 
we rewrite \eqref{eq:genericnetwork_linear}-\eqref{eq:genericnetwork} in terms of the convergence error defined for each node as $e_i=\left[e_{i}^{(1)},\ldots,e_{i}^{(n)}\right]^T$ with
\begin{equation}
 e_i=x_i-\bar{x},\label{eq:error},
\end{equation}
$\bar{x}$ being the \emph{average (node) trajectory} defined by
\begin{equation}
 \bar{x}=\frac{1}{N}\sum_{j=1}^N x_j.\label{eq:average_sys} 
\end{equation}
Using \eqref{eq:error} and \eqref{eq:average_sys}, from \eqref{eq:genericnetwork_linear} we obtain
\begin{equation}
 \dot{e}_i=f_i(t,x_i)-\frac{1}{N}\sum_{j=1}^{N}f_j(t,x_j)-c\sum_{j=1}^N { a_{ij}\Gamma (e_i-e_j)},\label{eq:error_dyn2}
\end{equation}
while from \eqref{eq:genericnetwork} we have
\begin{equation}
 \dot{e}_i=f_i(t,x_i)-\frac{1}{N}\sum_{j=1}^{N}f_j(t,x_j)-c\sum_{j=1}^N { a_{ij} \eta(t,e_i-e_j)},\label{eq:error_dyn1}
\end{equation} 

\begin{defi}
We say that network \eqref{eq:genericnetwork} (or \eqref{eq:genericnetwork_linear}) exhibits $\epsilon$\emph{-bounded convergence} iff
\begin{equation}\label{eq:lim_error_bounded}
\mathop{\lim}_{t\rightarrow\infty}\left|\left|e(t)\right|\right|_2\le\epsilon,
\end{equation}
with $e(t)=[e_1^T(t),\ldots,e_N^T(t)]^T$, $\epsilon\in\mathbb{R}^+$ and $\left\|\cdot\right\|_2$ representing the usual Euclidean norm.\footnote{The use of the symbol $\lim$ in \eqref{eq:lim_error_bounded} is not intended in the classical sense of \textsl{limit}. By \eqref{eq:lim_error_bounded}, we mean that for all $\nu >0$ there exists a 
$t_\nu>0$ such that for all $t>t_\nu$ we have that $\| e(t)\|_2\leq \epsilon+\nu$. We remark that this does not imply the existence of the limit in a classical sense.} 
\end{defi}

In what follows, we often use a compact notation both for the network state equations \eqref{eq:genericnetwork} and \eqref{eq:genericnetwork_linear}, and for the network error equations \eqref{eq:error_dyn1} and \eqref{eq:error_dyn2}. To this aim, we introduce the stack vector $x:=[x_1^T,\ldots,x_N^T]^T$ of all node states. Furthermore, assuming the node vector fields are QUAD affine, we call $\Phi(t,x)=\left[h_1^T(t,x_1),\dots, h_N^T(t,x_N) \right]^T $ the stack vector of the QUAD components, $\Psi(t,x)=\left[g_1^T(t,x_1),\dots, g_N^T(t,x_N) \right]^T $ the stack vector of the Affine components, and $\Xi=-1_N\otimes \frac{1}{N}\sum_{j=1}^{N}f_j(t,x_j)$ the term taking into account the dynamics of the average state, with $\mathbf{1}_N$ being the vector of $N$ unitary entries.  In this way, equations \eqref{eq:genericnetwork} and the error equation \eqref{eq:error_dyn1} can be recast, respectively, as

\begin{equation}\label{eq:state_stack}
\dot{x}=\Phi(t,x)+\Psi(t,x)-c\mathrm{H}(t,x),
\end{equation}
\begin{equation}\label{eq:error_stack}
\dot{e}=\Phi(t,x)+\Psi(t,x)+\Xi(t,x)-c\mathrm{H}(t,e),
\end{equation}
with
\[
\mathrm{H}(t,x)=\left[\begin{array}{c}
\sum_{j=1}^N { a_{1j}\eta(t,x_1-x_j)} \\
\vdots \\
\sum_{j=1}^N { a_{Nj}\eta(t,x_N-x_j)}
\end{array}
\right].
\]

If we define the Laplacian matrix $L = \left[\ell_{ij}\right]$ as
\begin{equation*}
\ell_{ij}=
\begin{cases}
-a_{ij}, & \textrm{if $i\neq j$ and $(i,j)\in\mathcal{E}$}\\ 
0, & \textrm{if $i\neq j$ and $(i,j)\notin\mathcal{E}$}\\ 
\displaystyle \sum_{{k=1}\atop {k\ne i}}^{N}a_{ik}, & \textrm{if $i=j$}
\end{cases},
\end{equation*} 
where $\mathcal{E}$ is the set of all the network edges, then, in the case of networks with linear coupling, the state equation \eqref{eq:genericnetwork_linear} and the error equation \eqref{eq:error_dyn2} can be recast as
\begin{equation}\label{eq:state_stack_linear}
\dot{x}=\Phi(t,x)+\Psi(t,x)-c(L\otimes\Gamma)x,
\end{equation}
\begin{equation}\label{eq:error_stack_linear}
\dot{e}=\Phi(t,x)+\Psi(t,x)+\Xi(t,x)-c\left(L\otimes \Gamma\right)e.
\end{equation}

Before giving the main results in the next sections, we recall here a useful lemma and define matrix sets that will be used in the paper.
\begin{lemma}\label{lem}
(\cite{GoRo:01}, pp. 279-288)
\begin{enumerate}
\item The Laplacian matrix $L$ in a connected undirected network is positive semi-definite. Moreover, 
it has a simple eigenvalue at $0$ and all the other eigenvalues are positive.
\item the smallest nonzero eigenvalue $\lambda_2(L)$ of the Laplacian matrix satisfies
\[
 \lambda_2(L)=\min_{z^T 1_N=0,z\ne0}\frac{z^T L z}{z^T z}.
\]
\end{enumerate}
\end{lemma}

Finally, we define the sets $\mathcal{Q}$ and $\mathcal{PW}$, that will be used in the rest of the paper and whose relevance will be clarified through a set of numerical example in the following sections.
\begin{defi}\label{def:pw}
Given a vector field $f:\mathbb R^+\times \mathbb{R}^n\rightarrow \mathbb{R}^n$, let $\mathcal Q\subseteq \mathcal{D}^+$ be the (possibly empty) set of matrices such that, for every $P\in\mathcal Q$, there exists a diagonal matrix $W$ such that \eqref{eq:quad} is satisfied. We say that a pair of matrices $(P_i,W_i)$ belongs to the set $\mathcal{PW}$ if and only if $P_i\in\mathcal D^+$ and $W_i\in\mathcal D$, and \eqref{eq:quad} is satisfied for $P=P_i$ and $W=W_i$.
\end{defi}
In the following sections, we provide a set of sufficient conditions for $\epsilon$-bounded convergence. Specifically, in Section \ref{sec:totally}, we study the case of linearly coupled networks of nonidentical piecewise smooth systems. Then, we extend the results to the case of networks coupled through nonlinear protocols in Section \ref{sec:partially}.

\section{Convergence analysis for linearly coupled networks}\label{sec:totally}
We consider a network modeled by equation \eqref{eq:genericnetwork_linear}  of $N$ nonidentical piecewise smooth QUAD($P,W_i$) Affine systems, $i=1,\ldots,N$.
\begin{ass}\label{ass:negativenonidentical} 
$h_i(t,x_i)$ is QUAD($P,W_i$) with $W_i<0$, $\left|\left|{\mathop g\limits ^{\sim}}_i(t,x_i)\right|\right|_2<M_i$ for all ${\mathop g\limits ^{\sim}}_i(t,x)\in{\mathop g\limits _{\sim}}{}_i(t,x)$, $t\in \mathbb{R}^+$, $x\in\R^n$, $i=1,\dots,N$,
\[
\sup_{t\in[0,+\infty)\atop i=1,\dots,N}\left|\left|{\mathop h\limits ^{\sim}}_i(t,0)\right|\right|_2
\leq 
\bar{h}_0<+\infty,
\] 
for all ${\mathop h\limits^{\sim}}_i(t,0)\in{\mathop h\limits_{\sim}}{}_i(t,0)$, and all the systems share a nonempty common set $\mathcal{C}_{\mathcal D^+}\subseteq\mathcal D^+$ such that every $P\in\mathcal{C}_{\mathcal D^+}$ implies $W_i<0$ satisfying inequality \eqref{eq:quad}, for all $i=1,\dots,N$.
%
\end{ass}
We define $\xoverline[.75]{M}$ as
\begin{equation}
\xoverline[.75]{M}=\max_{i=1,\dots, N} M_i.\label{eq:Mbar}
\end{equation}

Before illustrating our result, we need to give the following definitions.
\begin{defi}\label{def:B}
Given that Assumption \ref{ass:negativenonidentical} holds, and considering a $n\times n$ matrix $Q\in\mathcal{C}_{\mathcal D^+}$, the non-empty set $B(Q)\subset\mathbb{R}^{Nn}$ is
\begin{equation}
 B(Q)=\left\{ x\in\mathbb{R}^{Nn}:||x||_2<-\frac{\sqrt{N}||Q||_2\left(\xoverline[.75]{M}+\bar{h}_0\right)}{w_{\mathrm{max}}(Q)}\right\},
\label{eq:abs_domain_x}
\end{equation}
where $\xoverline[.75]{M}$ is defined in \eqref{eq:Mbar} $w_{\mathrm{max}}(Q)=\max\limits_{i=1,\dots,N}\lambda_{\mathrm{max}}\left(W_i(Q)\right)$, with $W_i<0$ such that $(Q,W_i)\in\mathcal{PW}$.

Also, we define the matrix $Q^*$,and the scalar $h_{\mathrm{max}}\in\mathbb{R}^+$ as 
\begin{equation}
 Q^*=\mathop{\mathrm{argmax}}_{Q\in\mathcal{C}_{\mathcal{D}^+}} \frac{||Q||_2\left(\xoverline[.75]{M}+\bar{h}_0\right)}{w_{\mathrm{max}}(Q)},
\label{eq:Qstar}
\end{equation}
\begin{equation}
h_{\mathrm{max}}=\max_{{i=1,\ldots,N\atop z\in B(Q^*)}\atop t\in[0,+\infty)}\left\|{\mathop h\limits ^{\sim}}_i(t,z)\right\|_2, \qquad  \forall {\mathop h\limits ^{\sim}}_i(t,z) \in {\mathop h\limits _{\sim}}{}_i(t,z),\quad \forall i=1,\dots, N.
\label{eq:hmax}
\end{equation}
\end{defi}
Notice that, in what follows, we always refer to the case in which $ h_i$ does not diverge in the finite ball $B(Q^*)$, implying $h_{\max}<\infty$. For this reason, the set-valued function ${\mathop h\limits _{\sim}}{}_i(t,z)$ is bounded for all time instants $t\in\mathbb{R}^+$ and takes values in the ball $B(Q^*)$ of the origin.
Notice also that in \eqref{eq:abs_domain_x} and \eqref{eq:Qstar} we state explicitly the dependence of $w_{\mathrm{max}}$ on $Q$. Indeed, a choice of $Q$ generally implies the selection of suitable matrices $W_1(Q),\dots,W_N(Q)<0$ satisfying relation \eqref{eq:quad}. 

Here, we define 
\begin{equation}\label{eq:W_max}
W^{\max}=\mathrm{diag}\left\{\max_{i=1,\dots,N}\lambda_1(W_i), \dots, \max_{i=1,\dots,N}\lambda_N(W_i) \right\}<0.
\end{equation}
Notice that, in \eqref{eq:W_max}, $W^{\max}$ depends on the choice of $P$, as well as the matrices $W_1,\dots,W_N<0$, and that $(P,W^{\max})$ belongs to the set $\mathcal{PW}$.
Furthermore, we also define the pair of matrices $P^*$ and ${W^{\mathrm{max}}}^*$ as
\begin{equation}
(P^*,{W^{\mathrm{max}}}^*)=\mathop {\text{argmin}}_{{P\in\mathcal{C}_{\mathcal D^+}, \atop
(P,W^{\mathrm{max}}(P))\in\mathcal{PW}}}
 \frac{\sqrt{N}||P||_2(\xoverline[.75]{M} + h_{\mathrm{max}})}{m\left(c,P,W^{\mathrm{max}}(P) \right)}, 
\label{eq:Pstar}
\end{equation}
where the real function  $m(c, P, W^{\mathrm{max}})$ is defined as 
\begin{equation}
\displaystyle m(c, P, W^{\mathrm{max}}) =-\max \left\{\lambda_{\max}(W^{\max}_l)- c \lambda_2(L \otimes P_l \Gamma_l), \lambda_{\max}(W^{\max}_{n-l})\right\},
\label{eq:mcp}
\end{equation}
with $W^{\max}_l$ and $P_{l}$ being the $l\times l$ upper-left block of matrices $W^{\max}$ and $P$ respectively, while $W^{\max}_{n-l}$ is the  $(n-l)\times (n-l)$ lower-right block of matrix $W^{\max}$. 

Now, we are ready to give the main stability results for linearly coupled networks. Specifically, we focus on the case of diagonal inner coupling matrix, while the extension to the case of nondiagonal $\Gamma$ is encompassed in the study of nonlinear coupling functions. Henceforth, here we consider $\Gamma=\mathrm{diag}\{\gamma_i\}_{i=1}^n$. Without loss of generality, we assume
\begin{equation}
{\gamma _i} = \left\{ {\begin{array}{*{20}{c}}
   {{{\bar \gamma }_i} > 0\quad} \hfill & {i = 1, \ldots ,l,} \hfill  \\
   0 \hfill & {i = l + 1, \ldots ,n,} \hfill  \\
\end{array}} \right.
\label{eq:gamma_i}
\end{equation} 
with $l\in\{0,1,\ldots,n\}$.
To use a compact notation, we denote by $\Gamma_l$ the $l\times l$ upper-left block of matrix $\Gamma$.
\begin{thm}\label{th:totallynonidentical}
Network \eqref{eq:genericnetwork_linear} of $N$ QUAD(P,$W_i$) Affine systems satisfying Assumption \ref{ass:negativenonidentical}, with diagonal inner coupling matrix $\Gamma\ge 0$, achieves $\epsilon$-bounded convergence for any value of the coupling strength $c>0$, and an upper bound for $\epsilon$ is given by
\begin{equation}
 \bar\epsilon=\min{ \left\{ \bar\epsilon_1:=-\frac{2\sqrt{N}||Q^*||_2(\xoverline[.75]{M}+\bar{h}_0)}{w_{\mathrm{max}}(Q^*)}, \bar\epsilon_2:=\frac{\sqrt{N}||P^*||_2(\xoverline[.75]{M} + h_{\mathrm{max}})}{m(c,P^*,W^{\mathrm{max}}(P^*))} \right\}  },
\label{eq:epsilon}
\end{equation}
where the function $m$ is defined in \eqref{eq:mcp}, 
and $Q^*$, $h_{\mathrm{max}}$ and $P^*$  are defined in \eqref{eq:Qstar}, 
\eqref{eq:hmax}, and \eqref{eq:Pstar} respectively.
\end{thm}
\begin{proof}
The proof consists of two steps. Firstly, we show the existence of an invariant region for the state trajectories of the nodes. Then, we derive the upper bound on $\epsilon$ as a function of the coupling gain $c$.

\textbf{Step 1.}
Given equation \eqref{eq:state_stack_linear}, let us consider the quadratic function
\begin{equation}
U=\frac{1}{2}x^T(I_N\otimes Q)x,
\label{eq:lyap_x}
\end{equation}
where $Q\in\mathcal{C}_{\mathcal D^+}$.
The time derivative of $U$ along the trajectories of the network satisfies
\[
\dot{U}(x)\in{\mathop \mathcal{L} \limits_\sim}{}_{\mathcal{F}\left[
\chi_1\right]}U(x),
\]
where $\chi_1(t,x)=\Phi(t,x)+\Psi(t,x)-c\left(L\otimes \Gamma\right)x$.
Applying the sum rule reported in Section \ref{sec:pws}, we can write
\begin{equation}\label{eq:udot1}
\dot{U}(x)\in{\mathop \mathcal{L} \limits_\sim}{}_{\mathcal{F}\left[
\chi_1
\right]}U(x)\subseteq
{\mathop \mathcal{L} \limits_\sim}{}_{\mathcal{F}\left[\Phi\right]+
\mathcal{F}\left[\Psi\right]+
\mathcal{F}\left[\chi_\gamma\right]}U(x),
\end{equation}
where $\chi_\gamma(t,x)=-c\left(L\otimes \Gamma\right)x$.

Applying the consistency rule to the smooth coupling term $\chi_\gamma$, we can write\footnote{Here and in what follows, given a vector $y$ and a set-valued function ${\mathop f \limits_\sim}$ of coherent dimension, by $y^T{\mathop f \limits_\sim}$ we mean $\left\{ y^T{\mathop f \limits^\sim}, \forall {\mathop f \limits^\sim}\in{\mathop f \limits_\sim}\right\}$.}
\begin{align}
{\mathop \mathcal{L} \limits_\sim}{}_{\mathcal{F}\left[\Phi\right]+
\mathcal{F}\left[\Psi\right]+
\mathcal{F}\left[\chi_\gamma\right]}U(x) = \mathcal{U_L} = 
\left\{
x^T\left(I_N\otimes Q\right)\mathop \Phi\limits _{\sim}+x^T\left(I_N\otimes Q\right)\mathop \Psi\limits _{\sim}
 -c x^T\left(L\otimes Q\Gamma\right)x \right\}.\label{eq.udot2}
\end{align}
Now, adding and subtracting $x^T\left(I_N \otimes Q\right){\mathop \Phi\limits _{\sim}}{}_0$, where ${\mathop \Phi\limits _{\sim}}{}_0=\mathcal{F}[\Phi](t,0)$, and using the product rule, we obtain  
\begin{equation}\label{eq:udot3}
\mathcal{U_L}\subseteq
\mathcal{V_L}=
\left\{\sum_{i=1}^N x_i^TQ{\mathop h\limits _{\sim}}{}_i(t,x_i)+
x^T\left(I_N\otimes Q\right)\mathop \Psi\limits _{\sim}-
cx^T\left(L\otimes Q\Gamma\right)x+
x^T\left(I_N \otimes Q\right){\mathop \Phi\limits _{\sim}}{}_0
 -\sum_{i=1}^N x_i^TQ{\mathop h\limits _{\sim}}{}_i(t,0)
\right\}.
\end{equation}
Therefore, using the QUAD assumption \eqref{eq:quad}, 
for a generic element of the set $v_l\in\mathcal{V_L}$, the following inequality holds
\begin{equation}\label{eq:udot4}
v_l\leq x^T\left[I_N\otimes w_{\mathrm{max}}(Q)I_n-c L \otimes Q\Gamma\right]x+
x^T\left(I_N \otimes Q\right)\mathop \Psi\limits ^{\sim}+x^T\left(I_N \otimes Q\right){\mathop \Phi\limits ^{\sim}}_0, \qquad \forall \mathop \Psi\limits ^{\sim}\in\mathop \Psi\limits _{\sim}, \forall {\mathop \Phi\limits ^{\sim}}_0\in{\mathop \Phi\limits _{\sim}}{}_0.
\end{equation}
From standard matrix algebra, we have (denoting $w_{\mathrm{max}}(Q)$ as $ w_{\mathrm{max}}$ for the sake of brevity)
\begin{equation}\label{eq:vl}
v_l\leq x^T\left[I_N\otimes w_{\mathrm{max}}I_n-c L \otimes Q\Gamma\right]x+
\left\|x\right\|_2\left\|I_N \otimes Q\right\|_2 \sqrt{N}\xoverline[.75]{M}+\left\|x\right\|_2\left\|I_N \otimes Q\right\|_2\sqrt{N}\bar{h}_0.
\end{equation}
Combining \eqref{eq:udot1}-\eqref{eq:vl}, it follows that
\[
\dot{U}\leq x^T\left[I_N\otimes w_{\mathrm{max}}I_n-cL \otimes Q\Gamma\right]x+
\left\|x\right\|_2\left\|I_N \otimes Q\right\|_2 \sqrt{N}\xoverline[.75]{M}+\left\|x\right\|_2\left\|I_N \otimes Q\right\|_2\sqrt{N}\bar{h}_0. 
\]
Rewriting the state vector as $x=a\hat x$, with $\hat x=\frac{x}{\left\| x \right\|_{2}}$, we finally have
\begin{equation}
\dot{U}\leq w_{\mathrm{max}}a^2+a\sqrt{N}||Q||_2\left(\xoverline[.75]{M}+\bar{h}_0\right).
\label{eq:lyap_der_final}
\end{equation}
Therefore, as $ w_{\mathrm{max}}<0$, if $a>- \sqrt{N}||Q||_2\left(\xoverline[.75]{M}+\bar{h}_0\right)/w_{\mathrm{max}}$, then $\dot{U}<0$. Hence, we can say that all the trajectories
of network \eqref{eq:state_stack_linear} eventually converge to the set $B(Q^*)$, where $B$ is given in Definition \ref{def:B} and $Q^*$ is defined 
in \eqref{eq:Qstar}. Thus, we can conclude that network \eqref{eq:genericnetwork_linear} achieves $\epsilon$-bounded convergence, with $\epsilon=-2\sqrt{N}||Q^*||_2\left(\xoverline[.75]{M}+\bar{h}_0\right)/w_{\mathrm{max}}$, being the bound on the convergence error.
Note that this estimate of the bound on $\epsilon$ might be conservative. We now derive an alternative bound.

\textbf{Step 2.} Let us consider equation \eqref{eq:error_stack_linear} and the following quadratic form
\begin{equation}
 V(e)=\frac{1}{2}e^T(I_N\otimes P)e,\nonumber
\label{eq:lyap_mism}
\end{equation}
where $P\in\mathcal{C}_{\mathcal{D^+}}$.
The time derivative of $V$ is
\begin{equation}\label{eq:v_dot_liederivative}
\dot{V}(e)\in{\mathop \mathcal{L} \limits_\sim}{}_{\mathcal{F}\left[
\chi_2
\right]}V(e),
\end{equation}
where $\chi_2(t,x,e)=\Phi(t,x)+\Psi(t,x)+\Xi(t,x)-c\left(L\otimes \Gamma\right)e$.
Using the sum and consistency rules, we obtain
\begin{align}\label{eq:lin_tot_nonid_baru_l}
\dot{V}(e)\in{\mathop \mathcal{L} \limits_\sim}{}_{\mathcal{F}\left[
\chi_2
\right]}V(e)\subseteq \mathcal{\overline U_L}=
\left\{
e^T\left(I_N\otimes P\right)\mathop \Phi\limits _{\sim}+e^T\left(I_N\otimes P\right)\mathop \Psi\limits _{\sim}+
e^T\left(I_N\otimes P\right)\mathop \Xi\limits _{\sim}
- e^T\left(L\otimes P\Gamma\right)e
\right\}. 
\end{align}
Now, from the properties of the Filippov set-valued function, and adding and subtracting $e^T\left(I_N \otimes P\right){\mathop \Phi\limits _{\sim}}{}_{\bar{x}}$, with ${\mathop \Phi\limits _{\sim}}{}_{\bar{x}}=\mathcal{F}[\Phi](t,\bar{x})$, and using the product rule we can write $\mathcal{\overline{U}_L}\subseteq \mathcal{\overline{V}_L}$, where $\mathcal{\overline{V}_L}$ is 
\begin{align*}
\mathcal{\overline V_L}= 
\left\{\sum_{i=1}^N e_i^TP{\mathop h\limits _{\sim}}{}_i(t,x_i)+
e^T\left(I_N \otimes P\right)\mathop \Psi\limits _{\sim}+
\sum_{i=1}^N e_i^TP\mathop \xi\limits _{\sim}-
ce^T\left(L\otimes P\Gamma\right)e+
e^T\left(I_N \otimes P\right){\mathop \Phi\limits _{\sim}}{}_{\bar{x}}-
\sum_{i=1}^N e_i^TP{\mathop h\limits _{\sim}}{}_i(t,\bar{x}), \right\},
\end{align*}
with $\mathop \xi\limits _{\sim}\in\mathcal{F}\left[-\frac{1}{N}\sum_{j=1}^N f_j(t,x_j)\right]$.
As $\sum_{i=1}^N e_i=0$, we have $\sum_{i=1}^N e_i^TP\mathop \xi\limits _{\sim}=0$. Considering the QUAD Affine assumption, a generic element $v_l$ of the set $\mathcal{V_L}$ satisfies the following inequality:
\begin{equation}\label{eq:inequality_v_e_dot_bounded}
v_l\leq e^T\left[I_N\otimes W^{\mathrm{max}}-cL \otimes P\Gamma\right]e+
e^T\left(I_N \otimes P\right)\mathop \Psi\limits ^{\sim}+e^T\left(I_N \otimes P\right){\mathop \Phi\limits ^{\sim}}_{\bar{x}}, \qquad \forall \mathop \Psi\limits ^{\sim}\in\mathop \Psi\limits _{\sim},\forall {\mathop \Phi\limits ^{\sim}}_{\bar{x}}\in{\mathop \Phi\limits _{\sim}}{}_{\bar{x}}\nonumber
\end{equation}
From the properties of the norm, and for all initial conditions $x(0)$ chosen in the set $B(Q)$, we have
\begin{equation}
v_l\leq e^T\left[I_N\otimes W^{\mathrm{max}}-c L \otimes P\Gamma\right]e+
\left\|e\right\|_2\left\|I_N \otimes P\right\|_2\sqrt{N}\xoverline[.75]{M}+\left\|e\right\|_2\left\|I_N \otimes P\right\|_2\sqrt{N}h_{\mathrm{max}},\nonumber
\end{equation}
where $h_{\mathrm{max}}$ is defined in \eqref{eq:hmax}. Hence, we have that
\begin{equation}
\dot{V}(e)\in{\mathop \mathcal{L} \limits_\sim}{}_{\mathcal{F}\left[
\chi_2
\right]}V(e)\subseteq \overline{\mathcal{V}}_\mathcal{L},\nonumber
\end{equation}
and we can write
\begin{equation}\label{eq:final_inequality_totallynonidentical_linear}
\dot{V}(e)\leq e^T\left[I_N\otimes W^{\mathrm{max}}-cL \otimes P\Gamma\right]e+
\left\|e\right\|_2\left\|I_N \otimes P\right\|_2\sqrt{N}\left(\xoverline[.75]{M}+h_{\mathrm{max}}\right).
\end{equation}
From the properties of the Kronecker product \cite{hojo:87}, we have $\left\|I_N\otimes P\right\|_2=\left\|I_N\right\|_2\left\|P\right\|_2=\left\|P\right\|_2$. 
Now, notice that the error vector $e$ can be decomposed in two parts: one is related to the coupled state components, namely $\tilde{e}_l=\left[ e_{1}^{(1)},\ldots,e_{1}^{(l)},\ldots,e_{N}^{(1)},\ldots,e_{N}^{(l)}\right]^T$, and the other, denoted by $\tilde e_{n-l}=\left[ e_{1}^{(l+1)},\ldots,e_{1}^{(n)},\ldots,e_{N}^{(l+1)},\ldots,e_{N}^{(n)}\right]^T$, to the uncoupled components. Furthermore, we define $\bar{e}_i=\left[ e_{1}^{(i)},e_{2}^{(i)},\ldots,e_{N}^{(i)}\right]^T$.
So, from \eqref{eq:gamma_i}, we can rewrite \eqref{eq:final_inequality_totallynonidentical_linear} as
\begin{equation}
\dot{V}\le \sum_{i=1}^l [ w^{\max}_i\bar{e}_i^T\bar{e}_i-cp_i\bar\gamma_i\bar{e}_i^T L \bar{e_i}]+\sum_{i=l+1}^n w^{\max}_i\bar{e}_i^T\bar{e}_{i} + \left\|e\right\|_2\left\|I_N \otimes P\right\|_2\sqrt{N}\left(\xoverline[.75]{M}+h_{\mathrm{max}}\right) ,\nonumber
\end{equation}
where $w^{\max}_i$ are the diagonal entries of the diagonal matrix $W^{\max}$.
From Lemma \ref{lem} and from matrix algebra, we have
\begin{equation}
\dot{V}\le \left[ \lambda_{\mathrm{max}}(W^{\max}_l)-c\lambda_2(L\otimes P_l\Gamma_l)\right]\tilde e_l^T \tilde e_l + \lambda_{\mathrm{max}}(W^{\max}_{n-l})\tilde e_{n-l}^T\tilde e_{n-l}^T +\left\|e\right\|_2\left\|I_N \otimes P\right\|_2\sqrt{N}\left(\xoverline[.75]{M}+h_{\mathrm{max}}\right).
\nonumber
\end{equation}
Then, rewriting the convergence error as $e=a\hat e$, with $\hat e=\frac{e}{\left\| e \right\|_{2}}$, for all initial conditions $x(0)\in B(Q)$ we finally obtain
\begin{equation}\label{eq:lyap_mism_final}
\dot{V}(e)\leq -m\left(c,P,W^{\mathrm{max}}\right)a^2+a\sqrt{N}\left\|P\right\|_2\left(\xoverline[.75]{M}+h_{\mathrm{max}}\right),
\end{equation}
with $m\left(c,P,W^{\mathrm{max}}\right)$ defined according to \eqref{eq:mcp}. Therefore, if $a>\sqrt{N}\left\|P\right\|_2\left(\xoverline[.75]{M}+h_{\mathrm{max}}\right)/m\left(c,P,W^{\mathrm{max}}\right)$, then $\dot V <0$.
From \eqref{eq:lyap_mism_final}, the optimization problem \eqref{eq:Pstar} immediately follows.
The minimum value of the bound in \eqref{eq:epsilon}, with $Q^*$ defined in \eqref{eq:Qstar}, is trivially obtained by combining \eqref{eq:lyap_der_final} and \eqref{eq:lyap_mism_final}.
\end{proof}
\begin{rema}
Notice that in the case where $h_i=h_j$ for all $i,j$ (which implies $W_i=W<0$ for all $i=1,\ldots,N$), $\epsilon$-bounded convergence is trivially guaranteed under the assumptions of Theorem \ref{th:totallynonidentical}, as the QUAD component of each system is contracting \cite{losl:98}, as reported in \cite{dedi:11}. In particular, asymptotic convergence ($\epsilon=0$) is achieved if $g_i=0$ for all $i=1,\ldots,N$, even if the systems are decoupled.
\end{rema}

Now, we study the stability properties of a networks of QUAD Affine systems, which differ only for the bounded component $g$. In this case we relax the assumption made earlier to prove Theorem \ref{th:totallynonidentical} and assume instead the following.
\begin{ass}\label{ass:partiallynonidentical}
Let us consider $N$ nonidentical piecewise smooth QUAD(P,W) Affine systems described by
\begin{equation}\label{eq:quad_affine_i}
\dot{x}_i=h_i(t,x_i)+g_i(t,x_i)\qquad \forall i=1,\dots,N,
\end{equation}
where 
\[
h_i(t,s)=h_j(t,s),\qquad  \forall i,j=1,\ldots,N,
\]
with $s\in\R^n$ and  $t\in\R^+$.
Furthermore, we call $\xoverline[.75]{M}=\max_{i=1,\dots, N} M_i$, with $M_i$ such that $\left|\left|{\mathop g\limits ^{\sim}}_i(t,x)\right|\right|_2<M_i$, for all ${\mathop g\limits ^{\sim}}_i(t,x)\in{\mathop g\limits _{\sim}}{}_i(t,x)$ and for all $t>0$ and $x\in\R^n$.
\end{ass}
Notice that, differently from Assumption \ref{ass:negativenonidentical}, here we do not make any additional assumption on the matrix $W$ which characterizes the QUAD components.
Even though the matrix $W$ is in general undefined, some of its diagonal elements may be negative.

According with the definition of $\Gamma_l$ given in Section \ref{sec:totally}, we denote by $W_l$ the $l\times l$ upper-left block of matrix $W=\mathrm{diag}\{w_i\}_{i=1}^n$, by $P_{l}$ the $l\times l$ upper-left block of matrix $P$, and by $W_{n-l}$ the $(n-l)\times (n-l)$ lower-right block of $W$. Also, we define the set $\mathcal{PW}_d^l$ as follows:
\begin{defi}\label{def:pwdl}
Given a positive scalar $d$, $\mathcal {PW}_d^l\subseteq\mathcal{PW}$ is the subset of $\mathcal{PW}$ such that if $(P,W)\in\mathcal{PW}_d^l$, then 
$d\lambda_2(L\otimes P_l \Gamma_l)>\lambda_{\mathrm{max}}(W_l)$, where $L$ is the Laplacian matrix of network
\eqref{eq:genericnetwork_linear}.
\end{defi}
Now, we are ready to state the following theorem.
\begin{thm}\label{th:partiallynonidentical_partiallycoupled}
Consider the network \eqref{eq:genericnetwork_linear} of $N$ nonidentical QUAD(P,W) Affine systems satisfying Assumption \ref{ass:partiallynonidentical}. Without loss of generality, we assume the first $\bar{l}\in\left\{0,\ldots,n\right\}$ diagonal elements of $W$ to be non-negative, while the remaining $n-\bar{l}$ are negative. If the diagonal elements of matrix $\Gamma\in\mathcal{D}$ can be defined as in equation \eqref{eq:gamma_i}, with $l\ge\bar{l}$, then there always exists a $\bar c<\infty$ so that, for any coupling gain $c>\bar{c}$, the linearly coupled network \eqref{eq:genericnetwork_linear} achieves $\epsilon$-bounded convergence. Furthermore,
\begin{enumerate}
	\item a conservative estimate, say $\tilde{c}$, of the minimum coupling gain $\bar{c}$ ensuring bounded convergence is
\begin{equation}
 \tilde{c}=\min_{(P,W)\in\mathcal{PW}} c(P,W),\label{eq:c_ub_sub}
\end{equation}
where $\displaystyle c(P,W)=\max\left\{\frac{\lambda_{\mathrm{max}}(W_l)}{\lambda_2(L \otimes P_l\Gamma_l)},0\right\}$.
 \item for a given $c>\tilde c$, we can give the following upper bound on $\epsilon$
\begin{equation}
 \bar{\epsilon}=\min_{(P,W)\in\mathcal{PW}_{c}^l}\frac{\xoverline[.75]{M}\sqrt{N} ||P||_2}{m(c,P,W)},\label{eq:varepsilon_ub_sub}
\end{equation}
where $m(c, P, W)$ is a real function defined as 
\begin{equation}
\displaystyle m(c, P, W) =-\max \left\{ \lambda_{\mathrm{max}}(W_l)- c \lambda_2(L \otimes P_l \Gamma_l),\lambda_{\mathrm{max}}(W_{n-l})\right\}.
\label{eq:mcp_partiallynonidentical}
\end{equation}
\end{enumerate}
\end{thm}
\begin{proof}
See Appendix \ref{app:a}.
\end{proof}
Notice that the computation of bound \eqref{eq:c_ub_sub} requires the solution of the following optimization problem:
\begin{equation}
\min_{(P,W)\in\mathcal {PW}}\max\left\{\frac{\lambda_{\mathrm{max}}(W_l)}{\lambda_2(L \otimes P_l\Gamma_l)},0\right\}.
\label{eq:opt}
\end{equation}
Trivially, if $f$ is QUAD($P$,$W$) Affine for some $W<0$, the solution of the optimization problem \eqref{eq:opt} is $\tilde{c}=0$. 
Otherwise, if a matrix $W<0$ such that $h$ is QUAD($P$,$W$) with $P\in\mathcal D^+$ does not exist (this is the case, for instance, of the Lorenz and Chua's chaotic systems), then the optimization problem \eqref{eq:opt} is non-trivial and, since $\lambda_{\mathrm{max}}(W)>0$, it can be rewritten as
\begin{equation}
\min_{(P,W)\in\mathcal{PW}}\frac{\lambda_{\mathrm{max}}(W_l)}{\lambda_2(L \otimes P_l\Gamma_l)}.
\label{eq:opt_rew}
\end{equation}
This is a constrained optimization problem that in scalar form can be written as:
\begin{equation}
 \mathop {\min }\limits_{(P,W)\in\mathcal{PW}\atop i=1,\ldots,l} \displaystyle{\frac{{{{\max }_{i}}{w _i}}}{{{\lambda _2}(L){{\min }_i}{p_i}{\gamma _i}}}},
\end{equation}
and which can be easily solved using the standard routines for constrained optimization, such as, for instance, those included in the MATLAB optimization toolbox. 

\begin{rema}
Here, we discuss the meaning of the assumptions and bounds obtained in Theorem \ref{th:partiallynonidentical_partiallycoupled}. 
Firstly, notice that the assumption on the vector field implies that the uncoupled components of the state vector are associated to contracting dynamics of the individual nodes. The minimum coupling strength needed to achieve bounded convergence is the minimum coupling ensuring shrinkage of the coupled part of the nodes' dynamics. 
Hence, the coupling configuration compensates for possible instabilities associated to positive diagonal elements of $W$. 
This minimum strength $\tilde c$ depends on the network topology. Specifically, the smaller $\tilde c$ is, the higher is $\lambda_2(L)$.
Once an appropriate coupling gain is selected, the width of the bound $\epsilon$ depends on $m(c,P,W)$
and on $\xoverline[.75]{M}$. Clearly, $\xoverline[.75]{M}$ gives a measure of the heterogeneity between the vector fields, and so the higher it is, the higher $\epsilon$ is. On the other hand, $m(c,P,W)$ 
embeds both the information on both the nodes' dynamics and the structure of their interconnections. In particular, the elements $p_i$ and $w_i$ of matrices $P$ and $W$, respectively, are related to the nodes' dynamics, while the information on the network topology are again embedded in $\lambda_2(L)$.
\end{rema}
When $\Gamma\in\mathcal{D}^+$, it is useful to consider the following corollary.
\begin{coro}\label{cor:partiallynonidentical}
Consider a network of $N$  QUAD(P,W) Affine systems satisfying Assumption \ref{ass:partiallynonidentical}. If the coupling matrix  $\Gamma\in \mathcal{D}^+$, then 
\begin{enumerate}
 \item there exists a $\bar c<\infty$ so that, for any coupling gain $c>\bar{c}$, network \eqref{eq:genericnetwork_linear} achieves $\epsilon$-bounded convergence.
 \item a conservative estimate, say $\tilde{c}$, of the minimum coupling gain ensuring $\epsilon$-bounded convergence is
\begin{equation}
 \tilde{c}=\min_{(P,W)\in\mathcal{PW}} c(P,W),\label{eq:c_ub}
\end{equation}
where $\displaystyle c(P,W)=\max\left\{\frac{\lambda_{\mathrm{max}}(W)}{\lambda_2(L \otimes P\Gamma)},0\right\}$, and $\mathcal{PW}$ is defined according to Definition \ref{def:pw}.
 \item for a given $\hat{c}>\tilde c$, we can give the following upper bound on $\epsilon$
\begin{equation}
 \bar{\epsilon}=\min_{(P,W)\in\mathcal{PW}_{\hat{c}}^n}\frac{\xoverline[.75]{M}\sqrt{N}\left\| P\right\|_2}{\hat{c}\lambda_2(L \otimes P\Gamma)-\lambda_{\mathrm{max}}(W)},\label{eq:varepsilon_ub}
\end{equation}
where the set $\mathcal{PW}_{\hat{c}}^n$ is defined according to Definition \ref{def:pwdl}.
\end{enumerate}
\end{coro}
\begin{proof}
If $\Gamma\in\mathcal{D}^+$, then clearly $l=n\ge\bar{l}$ in the proof of Theorem 2 for any $W\in\mathcal{D}$ and from Theorem \ref{th:partiallynonidentical_partiallycoupled}, the thesis follows.
\end{proof}

\section{Convergence analysis for nonlinearly coupled networks}\label{sec:partially}
Now, we address the problem of guaranteeing $\epsilon$-bounded convergence of \eqref{eq:genericnetwork} with a nonlinear coupling function $\eta$. Specifically, the analysis is performed for nonlinear coupling functions satisfying the following assumption.
\begin{ass}\label{ass3}
The (possibly discontinuous) coupling function $\eta(t,z):\mathbb{R}^+\times\mathbb{R}^n\mapsto\mathbb{R}^n$ is component-wise odd ($\eta(-v)=\eta(v)$) and the following inequality holds
\begin{equation}\label{eq:nonlinear_coupling_inequality}
z^T \mathop \eta\limits ^{\sim}(t,z)\ge z^T \Upsilon z,\qquad\qquad\forall t\in\mathbb{R}^+, \forall z:\left\|z\right\|_2	\le e_{\max},\, \forall \mathop \eta\limits ^{\sim}(t,z)\in\mathop \eta\limits _{\sim}(t,z),
\end{equation}
where $e_{\max}>0$ and $\Upsilon$ is a diagonal matrix whose $i$-th diagonal element is $\upsilon_i\ge0$, with $\sum_{i=1}^N=\upsilon_i>0$. Without loss of generality, we consider $\upsilon_i>0$ for all $i\le r$, with $r\le n$, while $\upsilon_i=0$ otherwise.
\end{ass}
Convergence to a bounded steady-state error is proved by assuming $Q=I$ and $P=I$. This choice, which is less general than the one considered in Theorem \ref{th:totallynonidentical}, allows however to analyze a more general nonlinear protocol. Following the same notation as in Section \ref{sec:totally}, we define $\Upsilon_r$ as the $r\times r$ upper left block of the matrix $\Upsilon$ in Assumption \ref{ass3}. Also, we define the scalars $r_{\max}$ and  $h_{\max}$ as
\begin{equation}
r_{\max}=\max\left\{ \bar{\epsilon}_1=
-\frac{\sqrt{N}\left(\xoverline[.75]{M}+\bar{h}_0\right)}{w_{\mathrm{max}}},\, \nu= \|x(0)\|_2
\right\}+\delta,\label{eq:rmax}
\end{equation}
with $\delta>0$ being a arbitrarily small positive scalar, and 
\begin{equation}\label{eq:h_max_totallynonidentical_nonlineare}
h_{\mathrm{max}}=\max_{{i=1,\ldots,N\atop \|z\|_2\leq r_{\max}}\atop t\in[0,+\infty)}\left\|{\mathop h\limits ^{\sim}}_i(t,z)\right\|_2, \qquad  \forall {\mathop h\limits ^{\sim}}_i(t,z) \in {\mathop h\limits _{\sim}}{}_i(t,z),\quad \forall i=1,\dots, N.
\end{equation}

\begin{thm}\label{th:totallynonidentical_nonlinear}
Consider the nonlinearly coupled network \eqref{eq:genericnetwork} of $N$ negative definite QUAD(I,$W_i$) Affine systems and suppose that the nonlinear coupling protocol satisfies Assumption \ref{ass3}.
Also, suppose that, in \eqref{eq:h_max_totallynonidentical_nonlineare}, $h_{\max}< +\infty$, and that each node of the network satisfies Assumption \ref{ass:negativenonidentical} with $P=I$. 
If
\begin{enumerate}
\item[(i)] The initial error satisfies $\|e(0)\|_2\leq {e_{\max}}/2$, with $e_{\max}$ defined in Assumption \ref{ass3};
\item[(ii)] 
\[
-\frac{\sqrt{N}(\xoverline[.75]{M} + h_{\mathrm{max}})}{\lambda_{\max}(W^{\max}_{n-r})}< \frac{e_{\max}}{2}
\]
where $W^{\max}$  is defined in \eqref{eq:W_max} and with $W^{\max}_{r}$ and $W^{\max}_{n-r}$ being  its upper-left and lower-right blocks, respectively;
\end{enumerate}
then, network \eqref{eq:genericnetwork} achieves $\epsilon$-bounded convergence if the coupling gain $c$ is chosen greater than $\tilde{c}$ given by
\begin{equation}\label{eq:c_tilde_totallynonidentical_nonlinear}
\tilde{c}= \max\left\{
\frac{1}{\lambda_2(L\otimes \Upsilon_r)}\left(\frac{2\sqrt{N}(\xoverline[.75]{M}+h_{\max})}{e_{\max}}+\lambda_{\max}(W^{\max}_r)\right),
0\right\}.
\end{equation}
Furthermore, an upper bound on $\epsilon$ is given by 
\begin{equation}\label{eq:bound_epsilon_totallynonidentical_nonlinear}
\bar\epsilon=\min{ \left\{ \bar\epsilon_1,
\bar\epsilon_2=\frac{\sqrt{N}(\xoverline[.75]{M} + h_{\mathrm{max}})}{m(c, W^{\max})} \right\}},
\end{equation}
with $\bar{\epsilon}_1$ defined as in \eqref{eq:rmax}, and
\[
\displaystyle m(c, W^{\max}) =-\max \left\{\lambda_{\max}(W^{\max}_r)- c \lambda_2(L \otimes \Upsilon_r), \lambda_{\max}(W^{\max}_{n-r})\right\}.
\]
\end{thm}
\begin{proof}
To prove the theorem, we separately analyze the two possible cases: $\bar{\epsilon}_1\leq\nu$ and $\bar{\epsilon}_1>\nu$, where $\nu$ is defined in \eqref{eq:rmax}.

\textit{Case (a): $\bar{\epsilon}_1\leq\nu$.}
\newline
In this case, from \eqref{eq:rmax} we have $r_{\max}=\|x(0)\|_2+\delta$.  
Now, we first study the conditions for the existence of an invariant region in the error space, and then show the existence of an invariant region in state space. We start by evaluating the derivative of the function $V(e)=\frac{1}{2}e^Te$.
We have
\begin{equation}
\dot{V}(e)\in{\mathop \mathcal{L} \limits_\sim}{}_{\mathcal{F}\left[
\chi_1
\right]}V(e),\nonumber
\end{equation}
where $\chi_1=\Phi(t,x)+\Psi(t,x)+\Xi(t,x)-c\mathrm{H}(t,e)$.
Using the sum rule, we can write
\begin{equation}
\dot{V}(x)\in \mathcal{\overline U_L} = 
\left\{
e^T\mathop \Phi\limits _{\sim}+e^T\mathop \Psi\limits _{\sim}+e^T\mathop \Xi\limits _{\sim}
-c e^T\mathop \mathrm{H}\limits _{\sim}\right\}.\label{eq:oveline_U}
\end{equation}
Adding and subtracting $e^T{\mathop \Phi\limits _{\sim}}{}_{\bar{x}}$, with ${\mathop \Phi\limits _{\sim}}{}_{\bar{x}}=\mathcal{F}[\Phi](t,\bar{x})$, and using the product rule, we have that
\begin{align}
\dot{V}(e)\in\mathcal{\overline U_L}\subseteq \mathcal{\overline V_L}= &
\left\{\sum_{i=1}^N e_i^T{\mathop h\limits _{\sim}}{}_i(t,x_i)+
e^T\mathop \Psi\limits _{\sim}+
\sum_{i=1}^N e_i^T\mathop \xi\limits _{\sim}-\frac{1}{2}
c\sum_{i=1}^N\sum_{j=i}w_{ij}(e_i-e_j)^T{\mathop \eta\limits _{\sim}}(t,e_i-e_j)+
e^T{\mathop \Phi\limits _{\sim}}{}_{\bar{x}}\right.\nonumber\\
&\left.-\sum_{i=1}^N e_i^T{\mathop h\limits _{\sim}}{}_i(t,\bar{x})
\right\},  \label{eq:tot_nonl_subset}
\end{align} 
with $\mathop \xi\limits _{\sim}\in\mathcal{F}\left[-\frac{1}{N}\sum_{j=1}^N f_j(t,x_j)\right]$.
As $\sum_{i=1}^N e_i=0$, we have $\sum_{i=1}^N e_i^T\mathop \xi\limits _{\sim}=0$.
As $\bar{\epsilon}_1<\nu$, inequality \eqref{eq:nonlinear_coupling_inequality} is satisfied for all $t\in [0, t_c]$, where $t_c$ is the time instant at which the average state trajectory may cross the ball of the origin of radius $r_{\max}$, i.e. $\|\bar{x}(t)\|_2>r_{\max}$ for $t>t_c$ (later we will show that such time instant does not exist and therefore \eqref{eq:nonlinear_coupling_inequality} is satisfied for all $t\in [0, +\infty)$ ). Indeed, from Assumptions \ref{ass:negativenonidentical} and \ref{ass3}, we have that a generic element of the set $v_l\in\mathcal{\overline V_L}$ satisfies the following inequality
\begin{equation}\label{eq:v_l_totallynonidentical_nonlineare}
v_l\leq e^T\left[I_N\otimes W^{\mathrm{max}}-cL \otimes \Upsilon\right]e+
e^T\mathop \Psi\limits ^{\sim}+e^T{\mathop \Phi\limits ^{\sim}}_{\bar{x}},\qquad \forall t\in [0, t_c], \quad \forall \mathop \Psi\limits ^{\sim}\in\mathop \Psi\limits _{\sim},\forall {\mathop \Phi\limits ^{\sim}}_{\bar{x}}\in{\mathop \Phi\limits _{\sim}}{}_{\bar{x}},
\end{equation}
and so, decomposing the error $e$ as $\tilde{e}_r$ and $\tilde{e}_{n-r}$ as in the proof of Theorem \ref{th:totallynonidentical}, and following similar steps, we have that 
\begin{equation}\label{eq:final_totallynonidentical_nonlineare_errore}
\dot{V}(e)\leq -m\left(c,W^{\max}\right)a^2+a\sqrt{N}\left(\xoverline[.75]{M}+h_{\mathrm{max}}\right),\qquad \forall t\in [0, t_c].
\end{equation}
Therefore, since hypothesis (ii) holds, it is now clear that if $c> \tilde{c}$, then relation \eqref{eq:nonlinear_coupling_inequality} is feasible as the region $\|e\|_2\leq \bar{\epsilon}_2< e_{\max}/2$ is an invariant region in the error space.
The feasibility of relation \eqref{eq:nonlinear_coupling_inequality} holds until the crossing instant $t_c$. After $t_c$, \eqref{eq:h_max_totallynonidentical_nonlineare} would not be guaranteed any more, as well as inequalities \eqref{eq:v_l_totallynonidentical_nonlineare} and \eqref{eq:final_totallynonidentical_nonlineare_errore}. 
To complete the proof of {\em Case (a)}, we now show that the crossing event never happens and so we can set $t_c=+\infty$.
Let us consider the quadratic function $U=\frac{1}{2}x^Tx$ and evaluate the derivative of $U$ along the trajectories of the network. We have  
\[
\dot{U}(x)\in{\mathop \mathcal{L} \limits_\sim}{}_{\mathcal{F}\left[
\chi_2
\right]}U(x),
\] 
where $\chi_2(t,x)=\Phi(t,x)+\Psi(t,x)-c\mathrm{H}(t,x)$.
Now, using the sum rule, and following similar steps as in Theorem \ref{th:totallynonidentical}, we can write
\[
\dot{U}(x)\in \mathcal{U_L} = 
\left\{
x^T\mathop \Phi\limits _{\sim}+x^T\mathop \Psi\limits _{\sim}
-c x^T\mathop \mathrm{H}\limits _{\sim} \right\}.
\]
Adding and subtracting $x^T{\mathop \Phi\limits _{\sim}}{}_0$, with ${\mathop \Phi\limits _{\sim}}{}_0=\mathcal{F}[\Phi](t,0)$, and using the product rule, we can show that $\mathcal{U_L}$ is included in the set $\mathcal{V_L}$. Namely,
\begin{align}\nonumber 
\mathcal{U_L}\subseteq
\mathcal{V_L}=
\left\{\sum_{i=1}^N x_i^T{\mathop h\limits _{\sim}}{}_i(t,x_i)+
x^T\mathop \Psi\limits _{\sim}-\frac{1}{2}
c\sum_{i=1}^N\sum_{j=i}^N w_{ij}(x_i-x_j)^T {\mathop \eta\limits _{\sim}}(t,x_i-x_j)+
x^T{\mathop \Phi\limits _{\sim}}{}_0-\sum_{i=1}^N x_i^T{\mathop h\limits _{\sim}}{}_i(t,0)
\right\}. 
\end{align}
Notice that, as stated above, relation \eqref{eq:nonlinear_coupling_inequality} holds for all the $t\in[0,t_c]$ and so, using Assumptions \ref{ass3} and \ref{ass:negativenonidentical}, 
for a generic element of the set $v_l\in\mathcal{V_L}$, the following inequality holds
\[
v_l\leq x^T\left[I_N\otimes w_{\mathrm{max}}I_n-c L \otimes \Upsilon\right]x+
x^T\mathop \Psi\limits ^{\sim}+x^T{\mathop \Phi\limits ^{\sim}}_0, \qquad \forall t\in[0,t_c],\quad \forall \mathop \Psi\limits ^{\sim}\in\mathop \Psi\limits _{\sim}, \forall {\mathop \Phi\limits ^{\sim}}_0\in{\mathop \Phi\limits _{\sim}}{}_0.
\]
Then, following the same steps in the proof of Theorem \ref{th:totallynonidentical} we obtain
\begin{equation}\label{eq:final_totallynonidentical_nonlineare_stato}
\dot{U}\leq w_{\mathrm{max}}a^2+a\sqrt{N}\left(\xoverline[.75]{M}+\bar{h}_0\right), \qquad \forall t\in[0,t_c].
\end{equation}
From \eqref{eq:final_totallynonidentical_nonlineare_stato}, we get the radius $\bar{\epsilon}_1$ of an invariant region $\|x\|_2\leq \bar{\epsilon}_1$ for system \eqref{eq:state_stack}. In particular, for any $r\geq \bar{\epsilon}_1$, the region $\|x\|_2\leq r$ is invariant. Since we are considering the case $\bar{\epsilon}_1\leq\nu$, then $\|x\|_2\leq r_{\max}$ is an invariant region for the overall system \eqref{eq:state_stack}. So, the state $x$, as well as $\bar{x}$, will never cross the ball of radius $r_{\max}$ and equations \eqref{eq:final_totallynonidentical_nonlineare_errore} and \eqref{eq:final_totallynonidentical_nonlineare_stato} hold with $t_c=+\infty$. Then comparing these two expressions, bound  \eqref{eq:bound_epsilon_totallynonidentical_nonlinear} holds and the proof for $\bar{\epsilon}_1\leq\nu$ is completed. 

\textit{Case (b): $\bar{\epsilon}_1>\nu$}
\newline
In this case, we have $r_{\max}=\bar{\epsilon}_1+\delta$. Again, we firstly consider the invariant region in the error space and then we analyze invariance in the state space. In particular, for the error invariant region we can follow the same steps of {\em Case (a)} and obtain again equation \eqref{eq:final_totallynonidentical_nonlineare_errore}. 
About the invariance in the state space, it is immediate to see that $\bar{\epsilon}_1$ is invariant. Indeed, if the trajectory $x(t)$ does not cross the boundary $\|x\|_2=\bar{\epsilon}_1$, then it is trivially invariant. On the other hand, if there exists an instant $\bar{t}$ such that $\|x(t)\|_2=\bar{\epsilon}_1$, then it is possible to show invariance of region $\|x\|_2\leq \bar{\epsilon}_1$ considering the proof of {\em Case (a)} from the initial time $\bar{t}$ and initial state $x(\bar{t}\,)$. 

Concluding, also in this case, equations \eqref{eq:final_totallynonidentical_nonlineare_errore} and \eqref{eq:final_totallynonidentical_nonlineare_stato} hold with $t_c=+\infty$ and the theorem is then proved.
\end{proof}
From Theorem \ref{th:totallynonidentical_nonlinear}, an useful corollary follows.
\begin{coro}\label{cor:totallynonidentical_nonlinear}
Consider the nonlinearly coupled network \eqref{eq:genericnetwork} of $N$ negative definite QUAD(I,$W_i$) Affine systems and suppose that the nonlinear coupling protocol satisfies Assumption \ref{ass3} with $e_{\max}=\infty$. Suppose also that each node of the network satisfies Assumption \ref{ass:negativenonidentical} with the choice $P=I$. 
Then, network \eqref{eq:genericnetwork} achieves $\epsilon-$bounded convergence and an upper bound on $\epsilon$ is \eqref{eq:bound_epsilon_totallynonidentical_nonlinear}. 
\end{coro}
\begin{proof}
As $e_{\max}\rightarrow +\infty$, the hypotheses {\em (i)}, {\em (ii)}, in Theorem \ref{th:totallynonidentical_nonlinear} are always satisfied and $\tilde{c}=0$. Then, from Theorem \ref{th:totallynonidentical_nonlinear} follows the thesis. 
\end{proof}
As in Section \ref{sec:totally}, we now extend the analysis to the case of networks \eqref{eq:genericnetwork} of QUAD Affine$(P,W)$ systems, with $P=I$, differing only for a bounded component. As in Theorem \ref{th:totallynonidentical_nonlinear}, we denote by $\Upsilon_r$ the $r\times r$ upper left block of matrix $\Upsilon$.
\begin{thm}\label{th:partiallinonidentical_nonlinear}
Let us consider the nonlinearly coupled network \eqref{eq:genericnetwork} of $N$  QUAD(I,$W$) Affine systems satisfying assumption \ref{ass:partiallynonidentical}.
Without loss of generality, we assume the first $\bar{r}\in\left\{0,\ldots,N\right\}$ diagonal elements of W to be non-negative, while the remaining $n-\bar r$ are negative. If Assumption \ref{ass3} holds with $r\ge \bar{r}$ and the following hypotheses hold:
\begin{enumerate}
\item[(i)] The initial error satisfies $\|e(0)\|_2\leq {e_{\max}}/2$, with $e_{\max}$ being defined in Assumption \ref{ass3};
\item[(ii)]
\[
-\frac{\sqrt{N} \xoverline[.75]{M} }{\lambda_{\max}(W_{n-r})}\leq \frac{e_{\max}}{2}
\]
with, as usual, $W_{r}$ and $W_{n-r}$ being the upper-left and the lower-right blocks of matrix $W$, respectively.
\end{enumerate}
Then, choosing a coupling gain $c> \tilde{c}$, with
\begin{equation}\label{eq:c_tilde_partiallynonidentical_nonlinear}
\tilde{c}= \max\left\{
\frac{1}{\lambda_2(L\otimes \Upsilon_r)}\left(\frac{2\sqrt{N}\xoverline[.75]{M}}{e_{\max}}+\lambda_{\max}(W_r)\right),
0\right\},
\end{equation}
network \eqref{eq:genericnetwork} achieves $\epsilon$-bounded convergence. Furthermore, an upper bound on $\epsilon$ is   
\begin{equation}
 \bar{\epsilon}=\frac{\xoverline[.75]{M}\sqrt{N}}{m_\mathrm{nl}(c,W)},\label{eq:varepsilon_ub_sub2}
\end{equation}
where the function $m_\mathrm{nl}$ is a real function defined as 
\begin{equation}
\displaystyle m_\mathrm{nl}(c, W) =-\max \left\{ \lambda_{\mathrm{max}}(W_r)- c \lambda_2(L \otimes \Upsilon_r),\lambda_{\mathrm{max}}(W_{n-r})\right\}.
\label{eq:mcp_partiallynonidentical2}
\end{equation} 
\end{thm}

\begin{proof}
See Appendix \ref{app:b}.
\end{proof}
As in Section \ref{sec:totally}, we also provide a useful corollary.
\begin{coro}
Let us consider the nonlinearly coupled  network \eqref{eq:genericnetwork} of $N$ QUAD(I,$W$) Affine systems satisfying Assumption \ref{ass:partiallynonidentical} and Assumption \ref{ass3} with $e_{\max}=\infty$. 
Choosing a coupling gain $\hat{c}\geq \tilde{c}$, with $\tilde{c}$ defined in \eqref{eq:c_tilde_partiallynonidentical_nonlinear}, network \eqref{eq:genericnetwork} achieves $\epsilon$-bounded convergence with an upper bound on $\epsilon$ given in \eqref{eq:varepsilon_ub_sub2}.   
\end{coro}
\begin{proof}
As $e_{\max}\rightarrow +\infty$, the hypotheses {\em (i)} and {\em (ii)} of Theorem \ref{th:partiallinonidentical_nonlinear} are always satisfied. Then, from Theorem \ref{th:partiallinonidentical_nonlinear}, the thesis follows.
\end{proof}

\section{Applications}
Here, we validate and illustrate the theoretical derivation using a set of representative numerical examples. Specifically, in Section \ref{sec:ikeda}, a network of Ikeda systems is considered to validate Theorems \ref{th:totallynonidentical} and \ref{th:totallynonidentical_nonlinear}, while Theorem \ref{th:partiallynonidentical_partiallycoupled} is used in Section \ref{sec:chua} to estimate the minimum coupling strength guaranteeing bounded synchronization in networks of Chua's circuits.
Then, in Section \ref{sec:relay}, Corollary \ref{cor:partiallynonidentical} is used to study convergence of coupled chaotic relays. Finally, in Section \ref{sec:kur} we study the convergence properties of nonuniform Kuramoto oscillators applying Theorem \ref{th:partiallinonidentical_nonlinear}.
\subsection{Networks of Ikeda systems}\label{sec:ikeda}
\begin{figure}
\centering
\subfigure[State evolution.]{\label{x_10ik}\includegraphics[width=8.6cm]{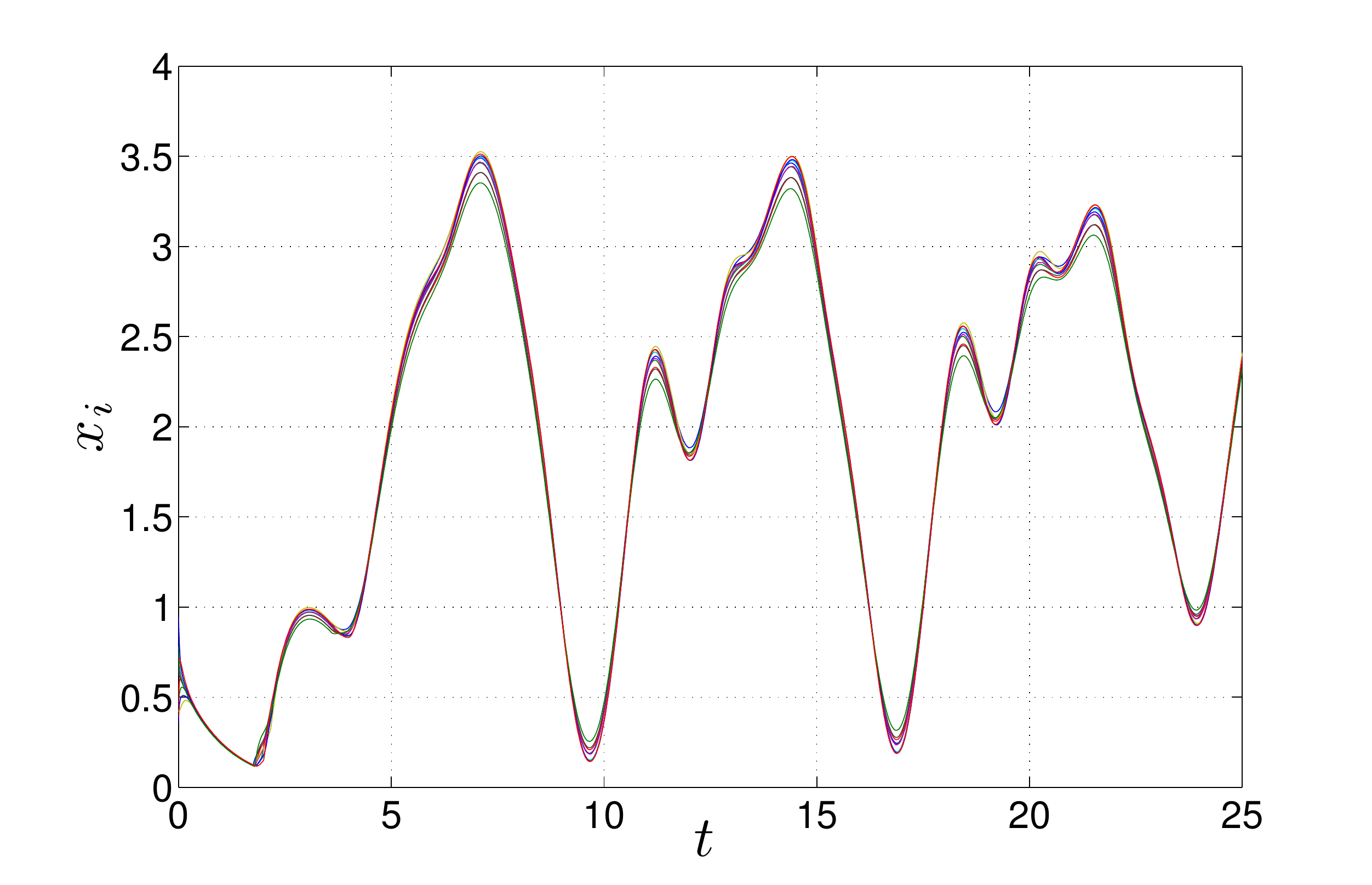}}
\subfigure[Norm of the error.]{\label{e_10ik}\includegraphics[width=8.6cm]{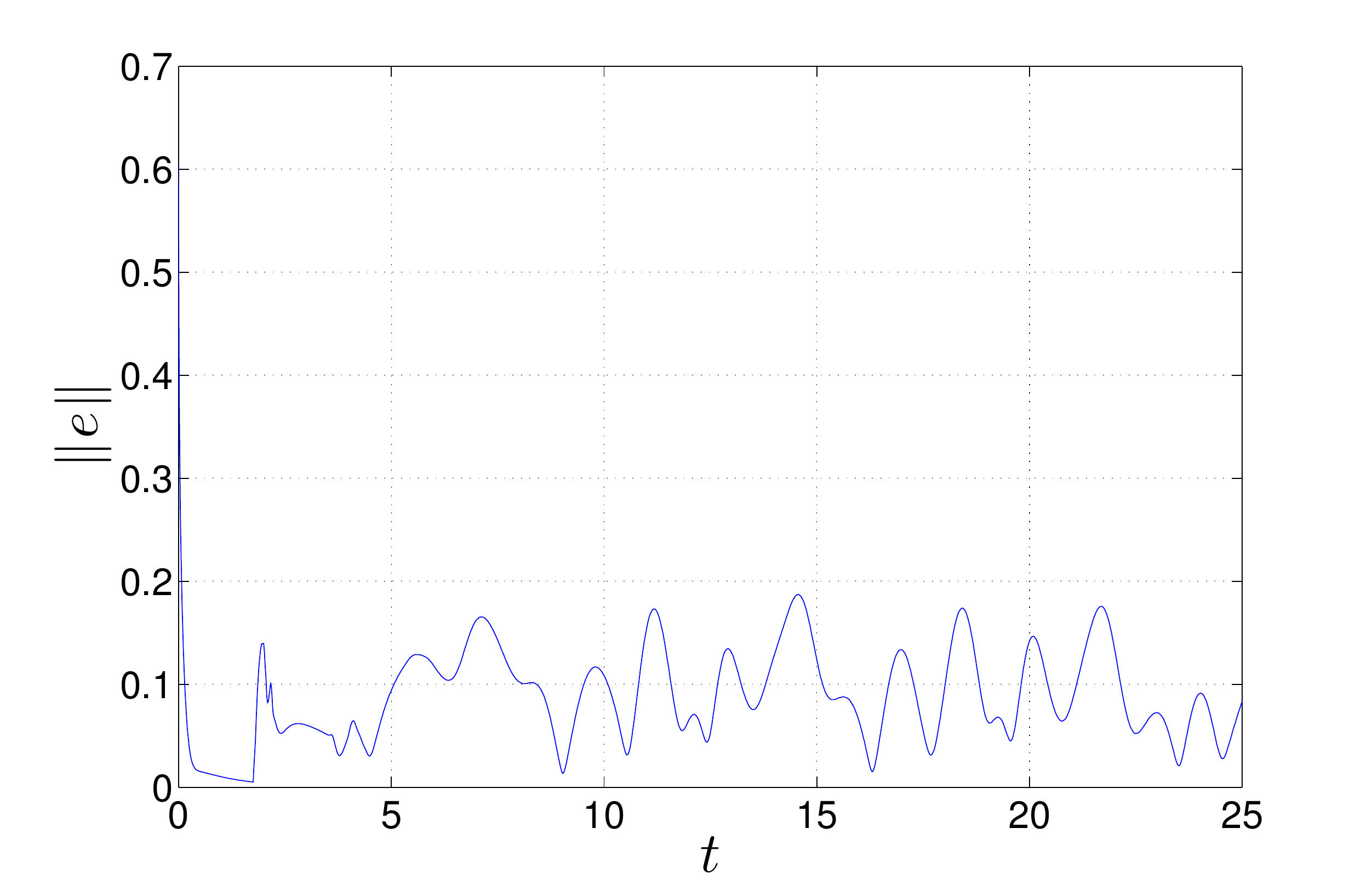}}
\caption{Network of $10$ linearly coupled nonidentical Ikeda systems. Coupling gain $c=20$.}
\label{ikeda10}
\end{figure}

\begin{figure}[h!]
  \centering 
  \includegraphics[width=9.6cm]{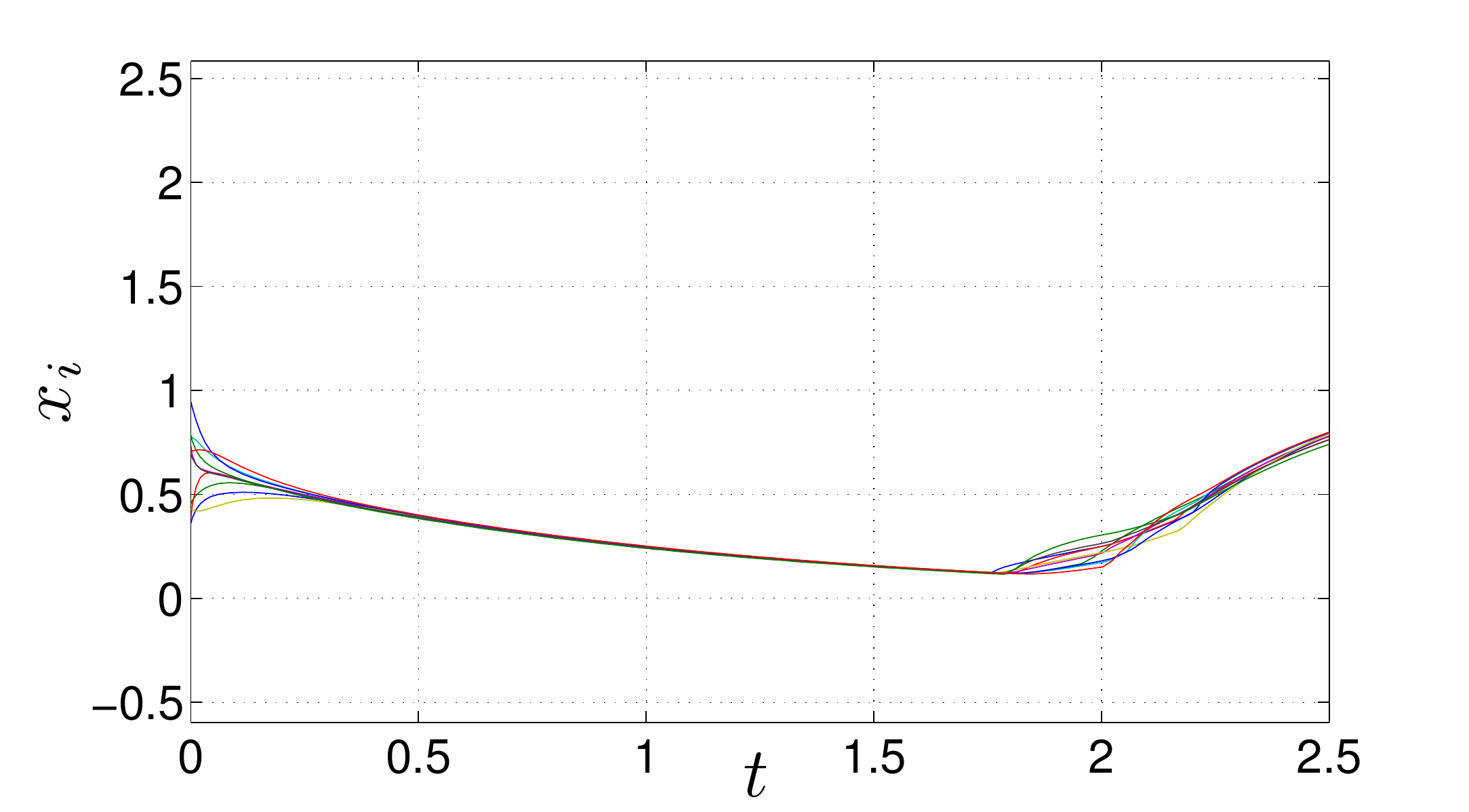}
  \caption{Network of $10$ linearly coupled nonidentical Ikeda systems, $c=20$: transient dynamics.}
  \label{fig:x_10ik_zoom_transient}
\end{figure}

To clearly illustrate Theorems \ref{th:totallynonidentical} and \ref{th:totallynonidentical_nonlinear}, we study the convergence of a network of nonidentical Ikeda systems.
The Ikeda model has been proposed as a standard model of optical turbulence in nonlinear optical resonators, see \cite{ik:79,ikda:80,ikma:86} 
for further details. The optical resonator can be described by
\begin{equation}
\dot{x}_i=-a_i x_i + b_i \sin (x_i(t-\tau_i)),\nonumber
\label{eq:ikeda}
\end{equation}
where $a_i$, $b_i$ and $\tau_i$ are positive scalars. As reported in \cite{huli:09}, this system exhibits chaotic behavior when $\tau_i = 2$, $a_i = 1$ and $b_i = 4$.
Synchronization of coupled Ikeda systems with parameter mismatches was studied in many recent works, see for instance \cite{huli:07,huli:09,shnu:08}, but it is assumed {\em a priori} that the trajectory of each node is bounded.
Applying Theorem \ref{th:totallynonidentical}, we do not need this assumption, and we can show that a network of coupled Ikeda oscillators converges to a bounded 
set. In facts, it is easy to show that the assumptions of Theorem \ref{th:totallynonidentical} are satisfied: the vector field $f_i(t,x_i)$ describing the nodes' dynamics is a QUAD(P,W) Affine system of the form
$$ \dot{x}_i=h_i(t,x_i)+g_i(t,x_i),$$
where $h_i(t,x_i)=-a_i x_i$ is QUAD with $P=p>0$ and $W=w$ such that $-pa_i\leq w <0$, and $g_i(t,x_i)=b_i \sin (x_i(t-\tau_i))$ is the affine bounded (smooth) term, with $|g_i(t,x_i)|\le b_i$. 
Notice that the presence of the delayed state does not prevent the application of Theorems \ref{th:totallynonidentical} and \ref{th:totallynonidentical_nonlinear}, as it affects a bounded component. 
Hence, from Theorem \ref{th:totallynonidentical}, we obtain a strong result: a network of nonidentical Ikeda systems is $\epsilon$-bounded synchronized for any possible value of the positive scalars $a_i$, $b_i$ and $\tau_i$, and for any positive coupling strength $c>0$, .
Here, it is worth remarking that this result is independent from the value of the delays $\tau_i$ and from the choice of $c$. In all previous works, $\tau$ was considered 
identical from node to node and bounded synchronization was proven only for $c>\bar c$, with $\bar c>0$ \cite{huli:07,huli:09,shnu:08}. Moreover, Theorem \ref{th:totallynonidentical} also provides an estimation of the bound $\epsilon$, that can be made arbitrarily small by increasing $c$.

\begin{figure}
\centering
\subfigure[Actual steady-state norm of the error.]{\label{eps_10ik}\includegraphics[width=8.6cm]{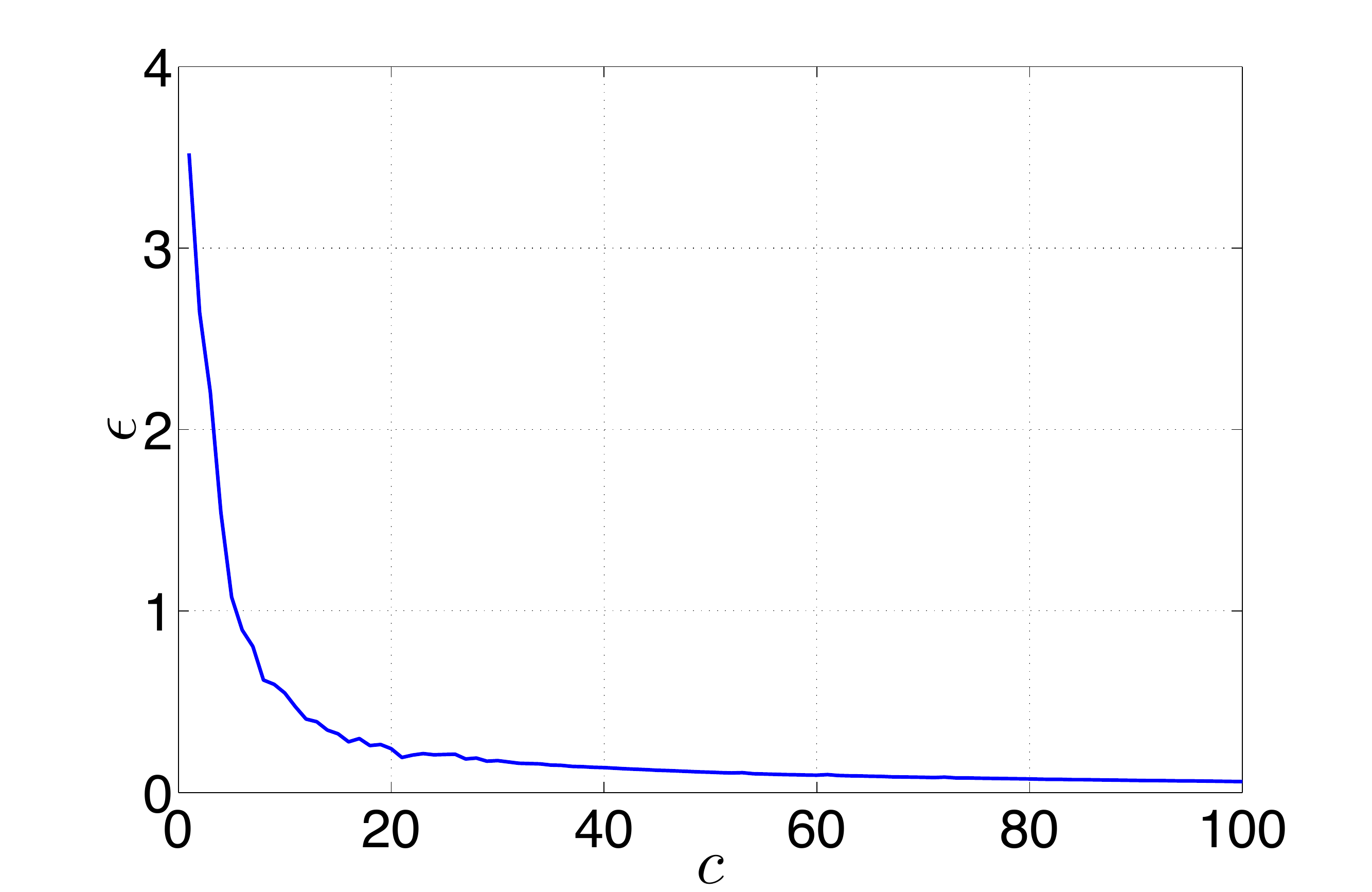}}
\subfigure[Upper bound for $\epsilon$, computed according to Theorem \ref{th:totallynonidentical}.]{\label{eps_tilde_10ik}\includegraphics[width=8.6cm]{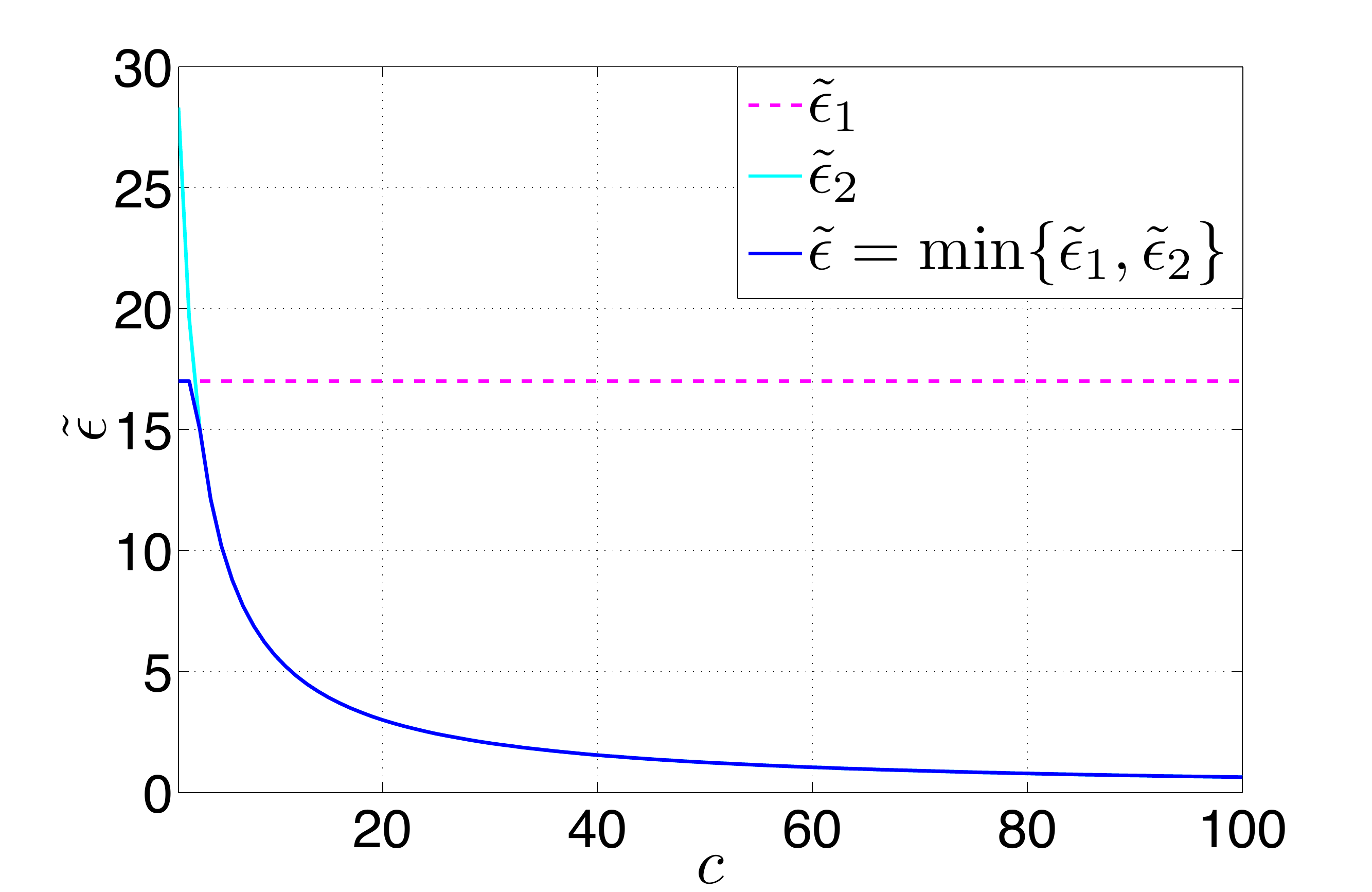}}
\caption{Network of $10$ nonidentical Ikeda systems.}
\label{ikeda10_eps}
\end{figure}

As an example, we consider a randomly generated network of 10 nodes. The initial conditions are taken randomly from a normal distribution.
Furthermore, we assume that $a_i=a+\delta a_i$, $b_i=b+\delta b_i$ and $\tau_i=\tau+\delta \tau_i$, where $a=1$, $b=4$, and $\tau=2$ are the 
nominal values of the parameters, while the parameters' mismatches are represented by $\delta a_i$, $\delta b_i$ and $\delta \tau_i$, and are taken randomly from a uniform distribution in $[-0.25,0.25]$.
As expected from the the theoretical predictions, the representative simulation with coupling gain $c=20$ shown in Figure \ref{ikeda10} confirms that $\epsilon$-bounded synchronization is achieved. In Figure \ref{fig:x_10ik_zoom_transient}, the onset of the state evolution is depicted to illustrate the transient dynamics.
Then, in Figure \ref{ikeda10_eps}, we report the upper bound for the steady-state error norm estimated for coupling strength $c$ ranging from $1$ to $100$ (Figure \ref{eps_tilde_10ik}), which is consistent with the maximum steady-state error norm evaluated numerically (Figure \ref{eps_10ik}). This upper bound is clearly conservative, but allows us to predict the exponential decay of $\epsilon$ as $c$ increases.

\begin{figure}[h!]
\centering
\subfigure[State evolution.]{\label{fig:ikeda_stateevolution_discontinuouscoupling_sgranata}\includegraphics[width=8.4cm]{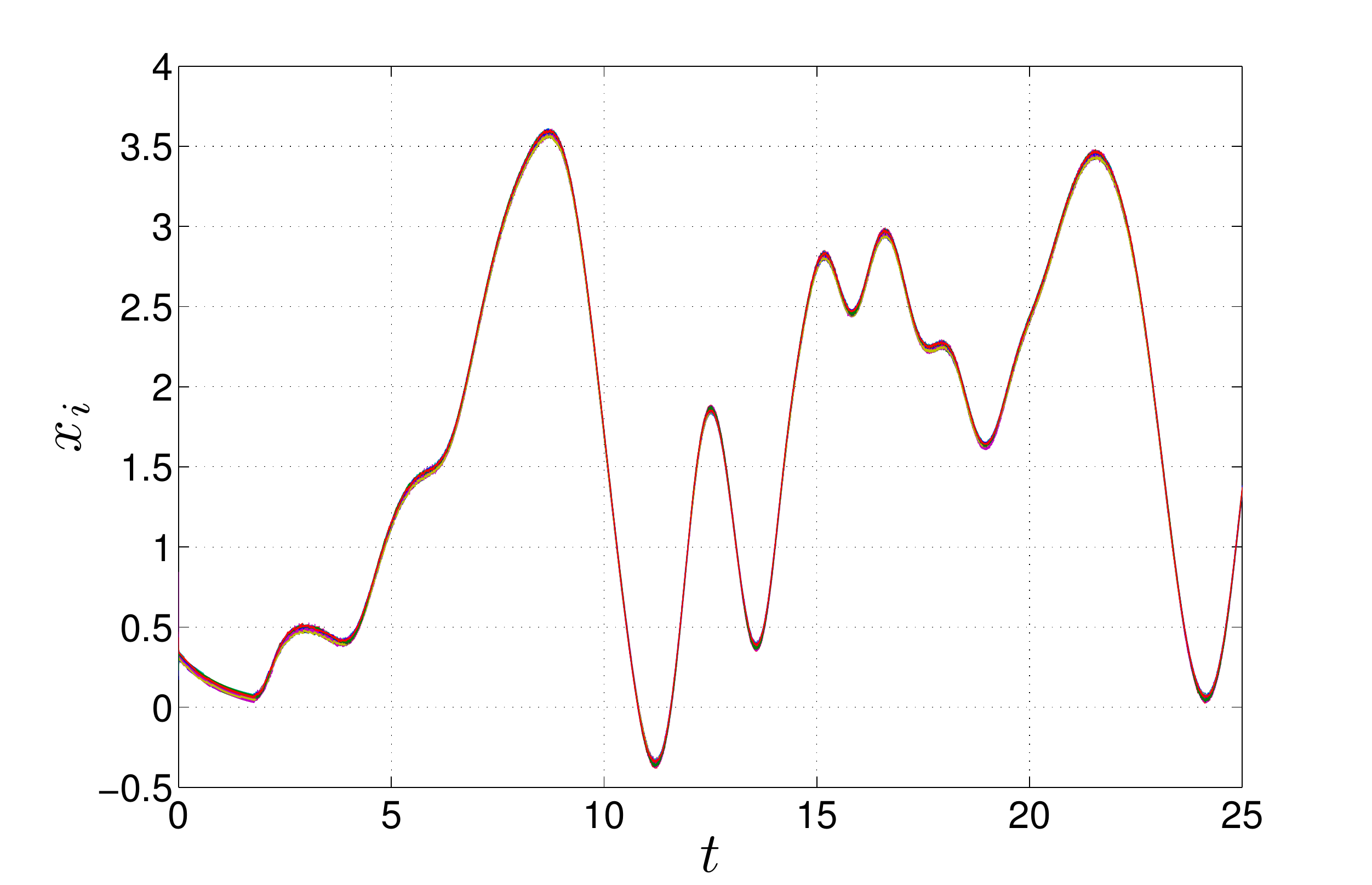}}
\subfigure[Norm of the error.]{\label{fig:ikeda_normerror_discontinuouscoupling_sgranata}\includegraphics[width=8.4cm]{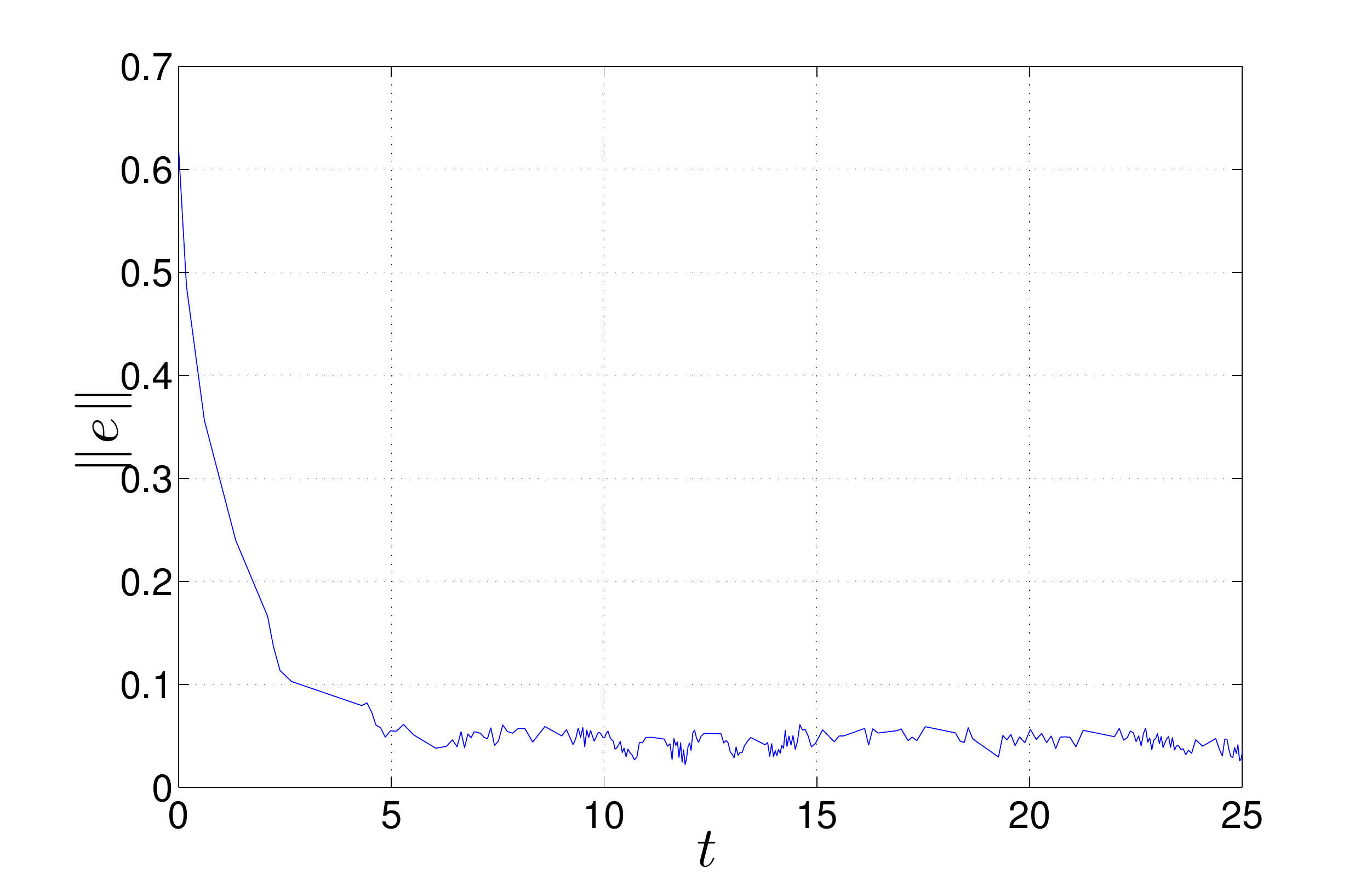}}
\caption{Network of $10$ nonlinearly coupled nonidentical Ikeda systems. Coupling gain $c=20$.}
\label{ikeda10_disc}
\end{figure}


Now, we consider a network of Ikeda systems with the same coupling gain, but we introduce the following piecewise smooth nonlinear coupling $\eta(z): \R\mapsto \R$:
\begin{equation}
\eta(z)=\begin{cases}
\mathrm{sign}(z) & \textit{ if }|z|<1, \\
\mathrm{sign}(z)[(|z|-1)^2+1] & \textit{ if }|z|\geq 1.
\end{cases}
\label{eq:nonl_coup_ik}
\end{equation}
This nonlinear coupling has no physical meaning, but has been introduced to show how $\epsilon$-bounded convergence can also be enforced through a piecewise smooth coupling satisfying Assumption \ref{ass3}.
Figures \ref{fig:ikeda_stateevolution_discontinuouscoupling_sgranata} and \ref{fig:ikeda_normerror_discontinuouscoupling_sgranata} confirm that bounded convergence is achieved considering the same coupling gain $c=20$.

\subsection{Networks of Chua's circuits}\label{sec:chua}
Let us consider now a network of  Chua's circuits \cite{Ma:84} -- a paradigmatic example often used in the literature on synchronization of nonlinear oscillators -- assuming each circuit is forced by a squarewave input. Namely, the own dynamics of the $i$-th system can be written as $\dot{x}_i=h(t,x_i)+g_i(t,x_i)$. 
The unforced dynamics are described by $h=[h_1,h_2,h_3]^T$. Namely,
\begin{align*}
{h}_{1}(x_i)&=\alpha\left[x_{i2}-x_{i1}-\varphi(x_{i1})\right],\\
{h}_{2}(x_i)&= x_{i1}-x_{i2}+x_{i3},\\
{h}_{3}(x_i)&=-\beta x_{i2},
\end{align*}
where, according to \cite{Ma:84}, $\alpha=10$, $\beta=17.30$, and $\varphi(x_{i1})=bx_{i1}+ (a-b)(\left|x_{i1}+1\right|-\left|x_{i1}-1\right|)/2$, with $a=-1.34$, $b=-0.73$. The squarewave input $g_i=[g_{i1},0,0]^T$ acts on the first variable and is defined as
$$g_{i1}(t)=\mathrm{sgn}(\sin(t-i\pi/N)).$$
Notice that the vector fields of the Chua's circuits are  nonidentical QUAD(P,W) Affine and satisfy Assumption \ref{ass:partiallynonidentical}. In fact, for any $P\in\mathcal{D}^+$, and for any $x,y\in\mathbb{R}^3$, we can write
\begin{align*}
(x-y)^T P (h(x)-h(y))&=-10p_1 e_1^2-p_2 e_2^2+(10 p_1+p_2)e_1e_2+(p_2-17.3p_3)e_2e_3+10p_1e_1(\varphi(y_1)-\varphi(x_1))\\
& \le  3.4 p_1 e_1^2-p_2 e_2^2+(10 p_1+p_2)e_1e_2+(p_2-17.3p_3)e_2e_3,
\end{align*}
where $e=x-y$, and where we have considered the maximum slope of the nonlinear function $\varphi(\cdot)$ to get the above inequality.
Taking $p_2=17.3p_3$, and being $e_1 e_2\le\left\|e_1e_2\right\|\le(\rho e_1^2+e_2^2/\rho)/2$ for all $\rho>0$, one has
\begin{align}
(x-y)^T P (h(x)-h(y))& \le 3.4 p_1 e_1^2-p_2 e_2^2+(10 p_1+p_2)e_1e_2\nonumber\\
& \le  3.4 p_1 e_1^2-p_2 e_2^2 + (10p_1+p_2)(\rho e_1^2+e_2^2/\rho)/2\nonumber\\
& = \left(3.4 p_1+\rho\frac{10 p_1+p_2}{2}\right)e_1^2+\left(\frac{10p_1+p_2}{2\rho}-p_2\right)e_2^2.\label{eq:chua_quad}
\end{align}
Moreover, for all $x\in\mathbb{R}^3$, $\left\|g(t,x)\right\| \le \xoverline[.75]{M}=1$. Therefore, one finally obtains that the forced Chua's circuit are QUAD(P,W) Affine systems for any pair (P,W) such that $p_2=17.3 p_3$ and $w_1\ge3.4 p_1+\rho(10p_1+p_2)/2$ and $w_2\ge -p_2 + (10p_1+p_2)/2\rho$. Therefore, it is possible to take $p_1$, $p_2$ and $\rho$ such that $w_2$ can be negative.
%
Hence, if we select $\Gamma=\mathrm{diag}\left\{1,0,1\right\}$, we have that all the assumptions of Theorem \ref{th:partiallynonidentical_partiallycoupled} are satisfied with $l=2$.\footnote{The application of Theorem \ref{th:partiallynonidentical_partiallycoupled} requires a trivial reordering of the state variables, that we omit here.}

Notice that Theorem \ref{th:partiallynonidentical_partiallycoupled} can be used to estimate the minimum coupling strength guaranteeing bounded synchronization. From \eqref{eq:c_ub_sub} and \eqref{eq:chua_quad} follows that the estimation $\tilde{c}$ is the solution of the following constrained optimization problem:
\begin{align*}
& \tilde{c}=\frac{1}{\lambda_2(L)}\min_{p_1,p_2,\rho}\frac{3.4 p_1+\rho\frac{10p_1+p_2}{2}}{\min\left\{p_1,\frac{p_2}{17.3}\right\}}\\
& \frac{10p_1+p_2}{2\rho}-p_2<0\\
&p_1,p_2,\rho>0
\end{align*}
where the first constraint ensures that $P$ is selected so that the system is QUAD(P,W), with $w_2<0$, thus allowing the selection of $\Gamma=\mathrm{diag}\left\{1,0,1\right\}$ according to Theorem \ref{th:partiallynonidentical_partiallycoupled}. 
Using the standard Matlab routines for constrained optimization problems, one easily obtains $\tilde{c}=14.17/\lambda_2(L)$.
In this example, we consider a network of $N=10$ nodes with a connected random graph \cite{erre:59} with $\lambda_2(L)=2.22$, from which follows that $\tilde{c}=6.4$. Accordingly, we select $c=10>\tilde{c}$. Figure \ref{fig:x1_caos} shows the time evolution for the first component of the Chua's oscillators, both for the uncoupled and the coupled case. 
It is possible to observe that a a reduced mismatch between the nodes' trajectories remains, as can be noted from the plot of the error norm, depicted in Figure \ref{fig:norma_errore_caos}. This simulation has been obtained considering random initial conditions in the domain of the chaotic attractor. However, it is worth mentioning that since Theorem \ref{th:partiallynonidentical_partiallycoupled} gives global synchronization conditions, bounded synchronization is ensured also in the case of divergent dynamics, as shown 
in Figure \ref{fig:norma_errore_divergenti}, where some initial conditions have been randomly chosen outside the domain of the attractor.
\begin{figure}[h!]
\centering
\subfigure[]{\label{fig:x1_nonaccoppiati_caos}\includegraphics[width=8.6cm]{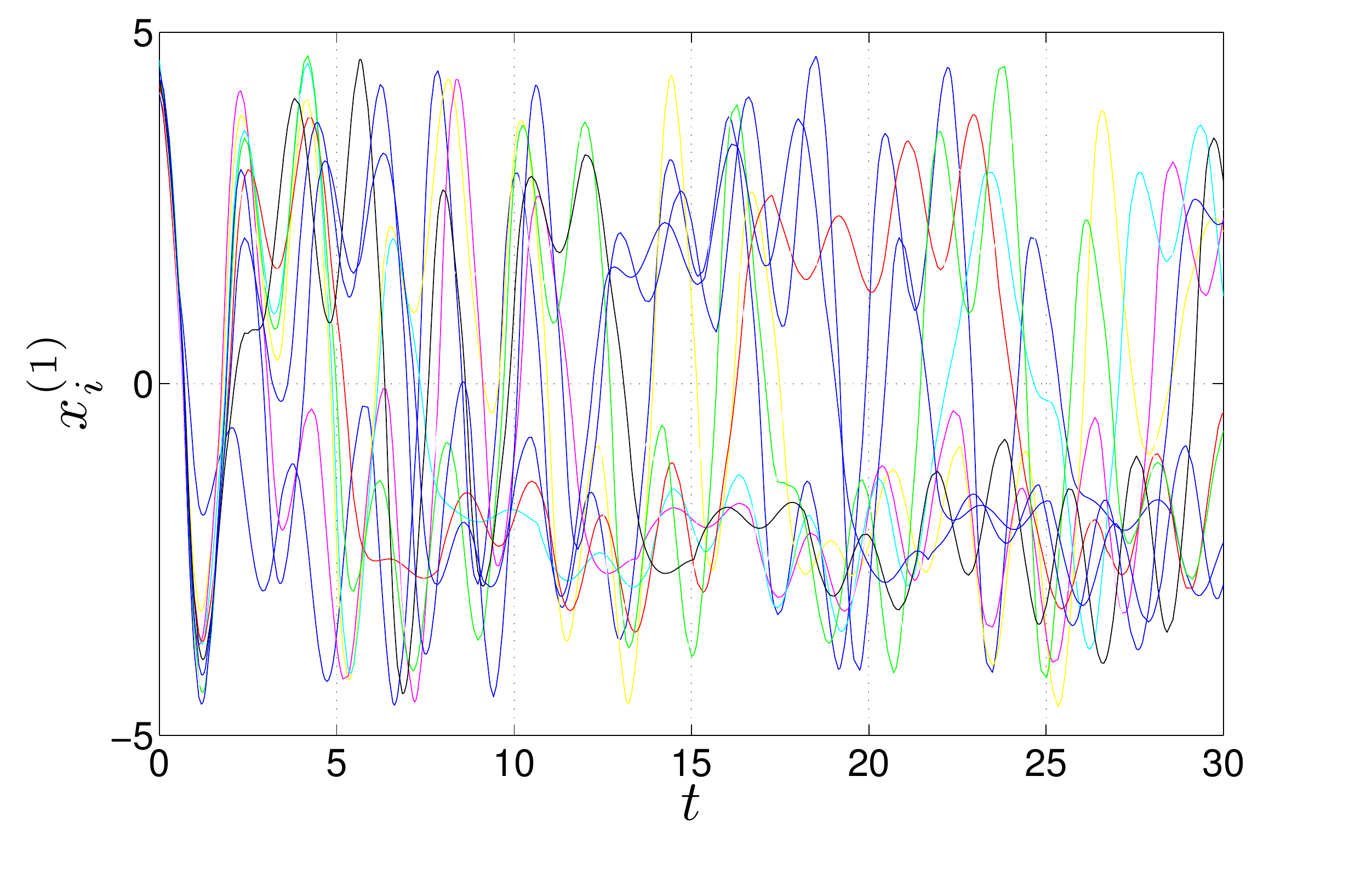}}
\subfigure[]{\label{fig:x1_accoppiati_caos}\includegraphics[width=8.6cm]{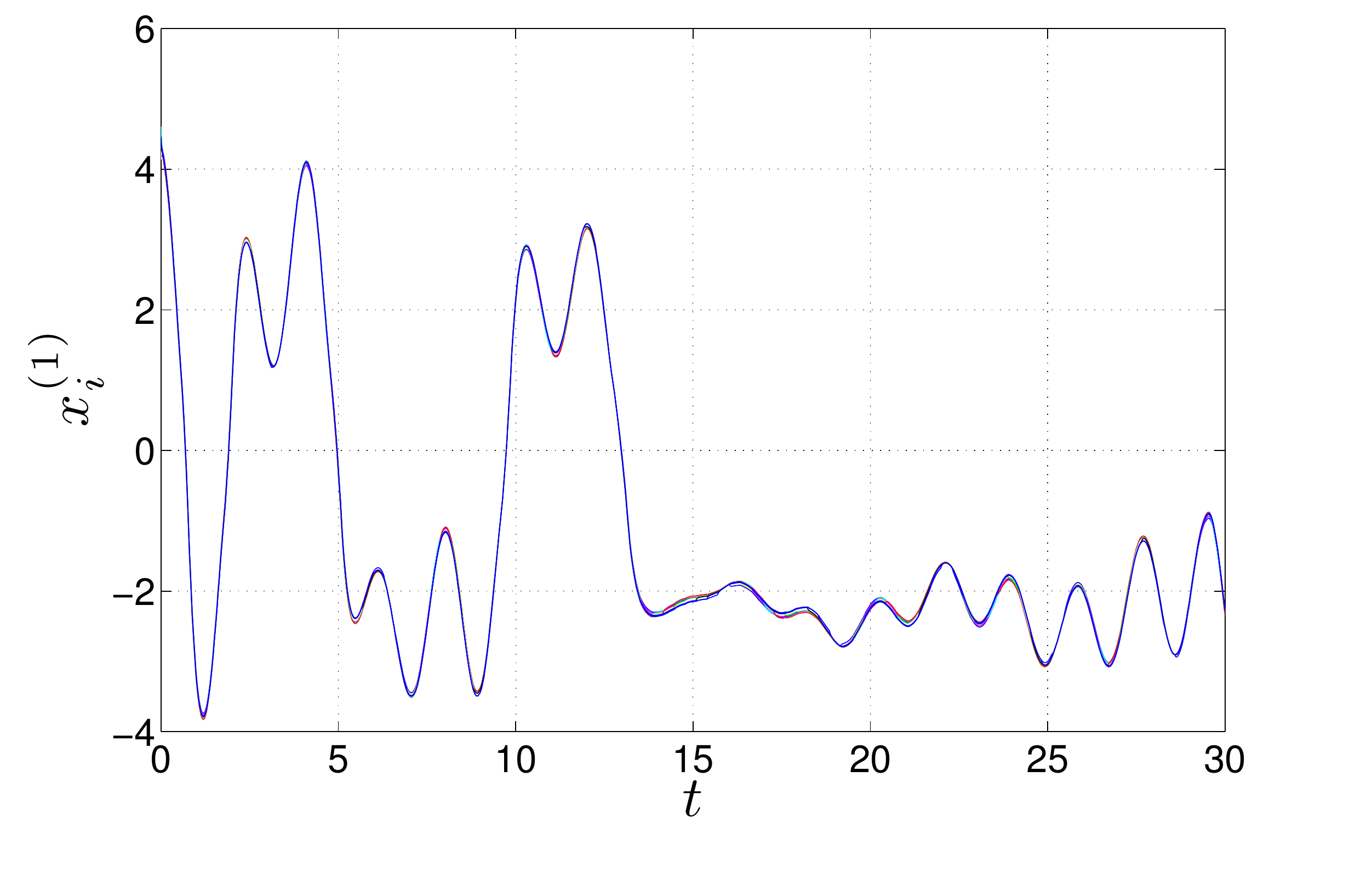}}
\caption{Time evolution of component $x_i^{(1)}(t)$ for the network of Chua circuits: (a) uncoupled case; (b) coupled case.}
\label{fig:x1_caos}
\end{figure}
\begin{figure}[h!]
\centering
\subfigure[]{\label{fig:norma_errore_nonaccoppiati_caos}\includegraphics[width=8.6cm]{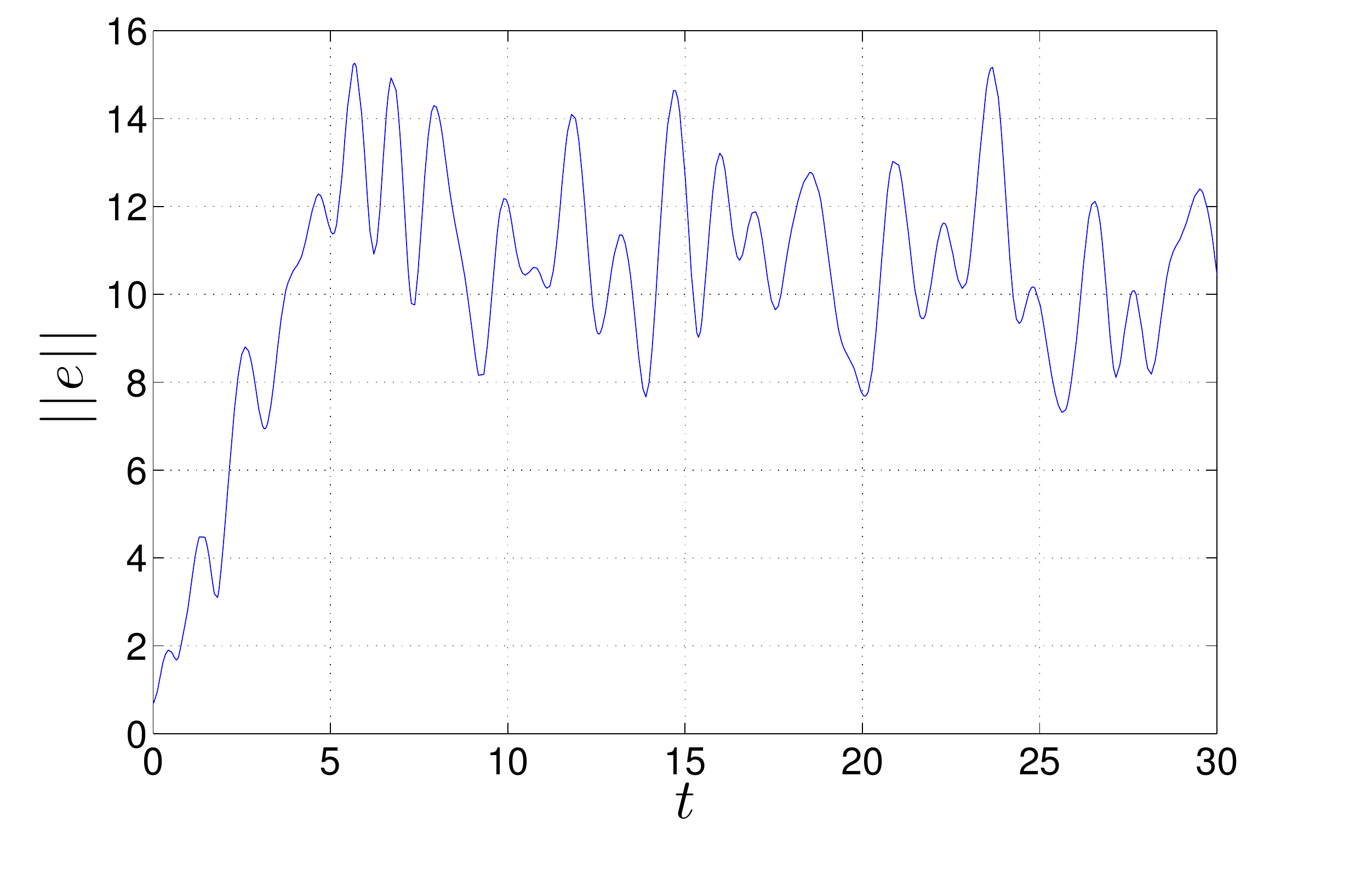}}
\subfigure[]{\label{fig:norma_errore_accoppiati_caos}\includegraphics[width=8.6cm]{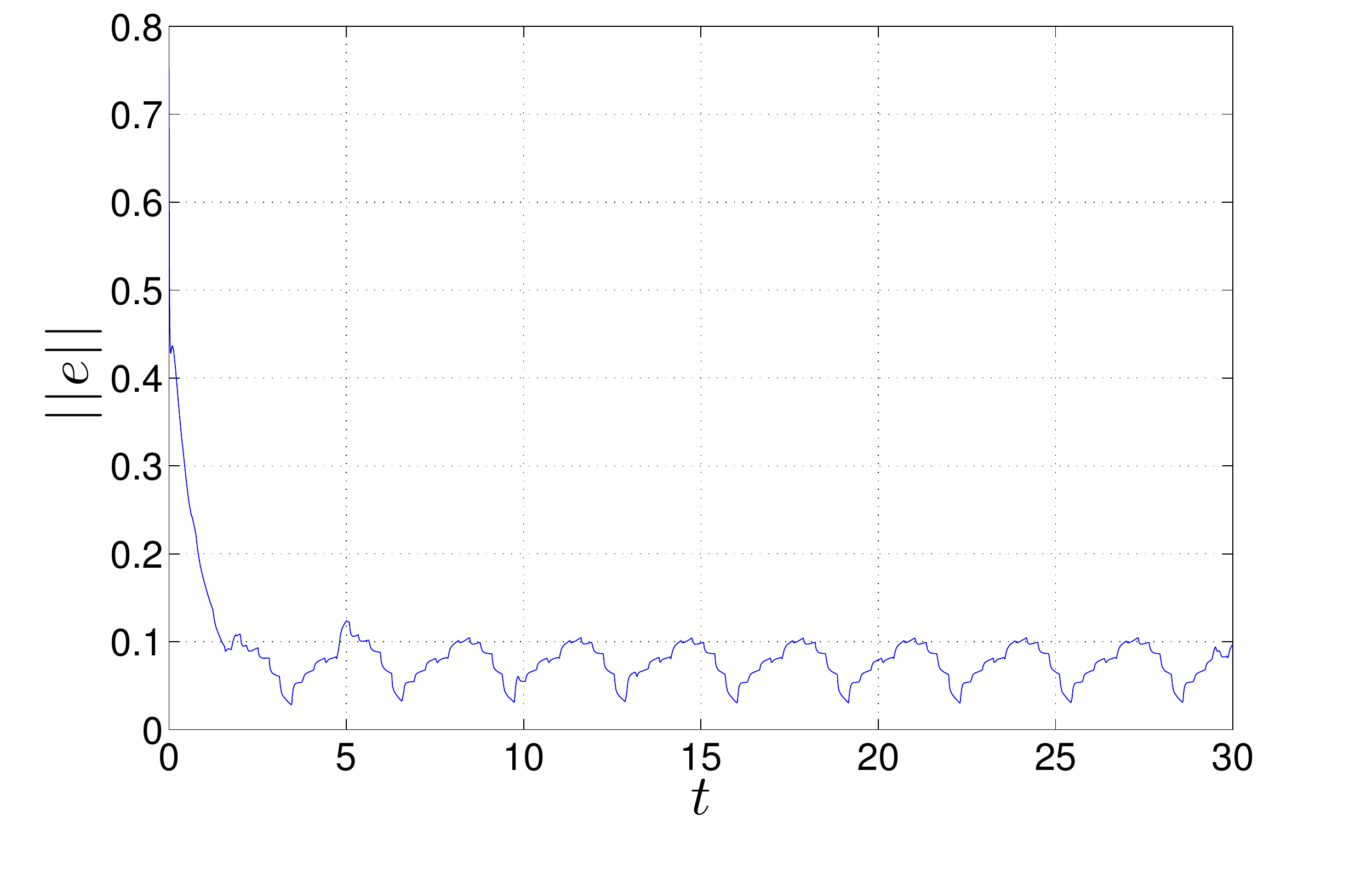}}
\caption{Time evolution of the norm of the synchronization error for the network of Chua circuits: (a) uncoupled case; (b) coupled case.}
\label{fig:norma_errore_caos}
\end{figure}
\begin{figure}[h!]
\centering
\subfigure[]{\label{fig:norma_errore_nonaccoppiati_divergenti}\includegraphics[width=8.6cm]{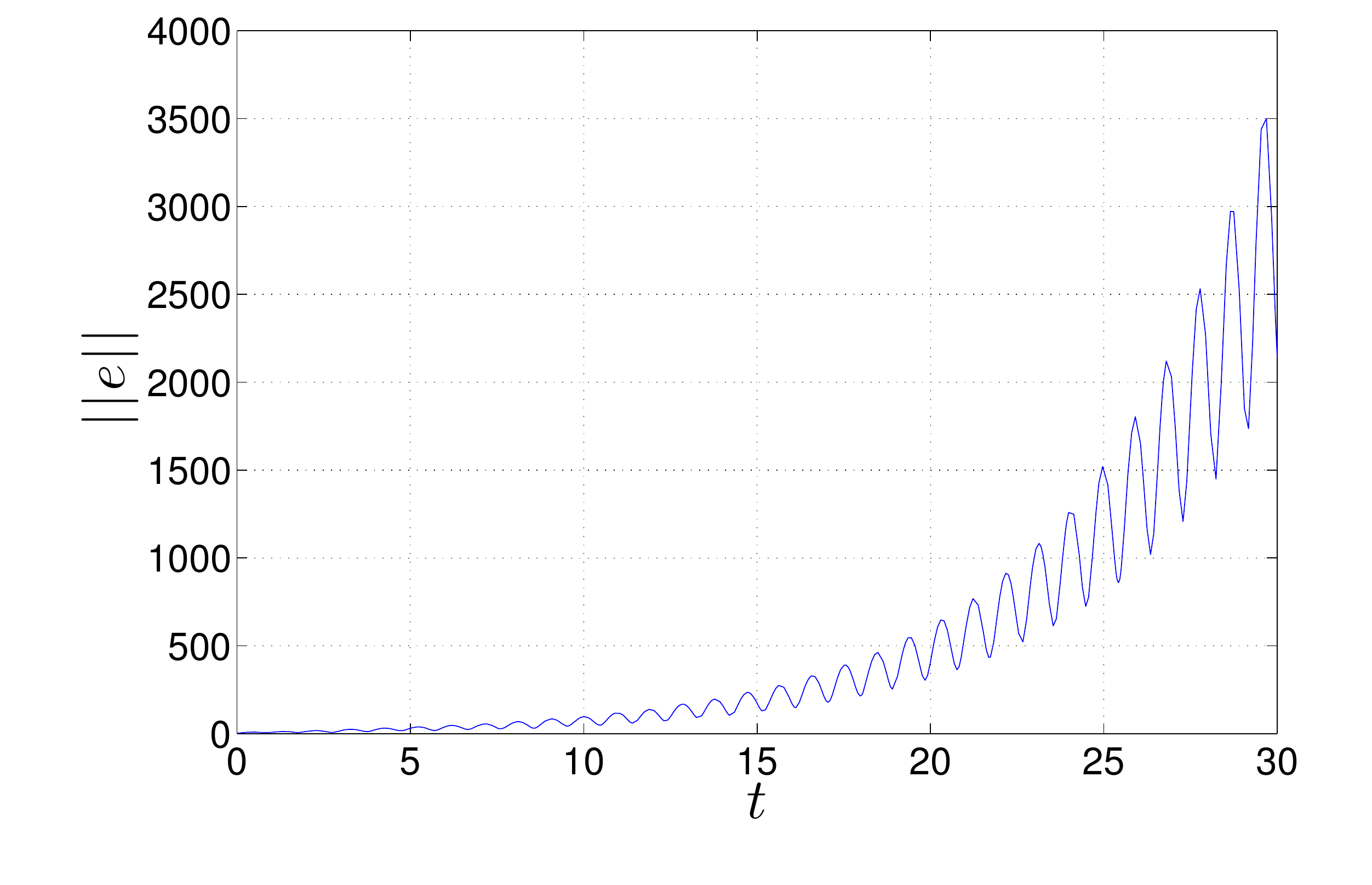}}
\subfigure[]{\label{fig:norma_errore_accoppiati_divergenti}\includegraphics[width=8.6cm]{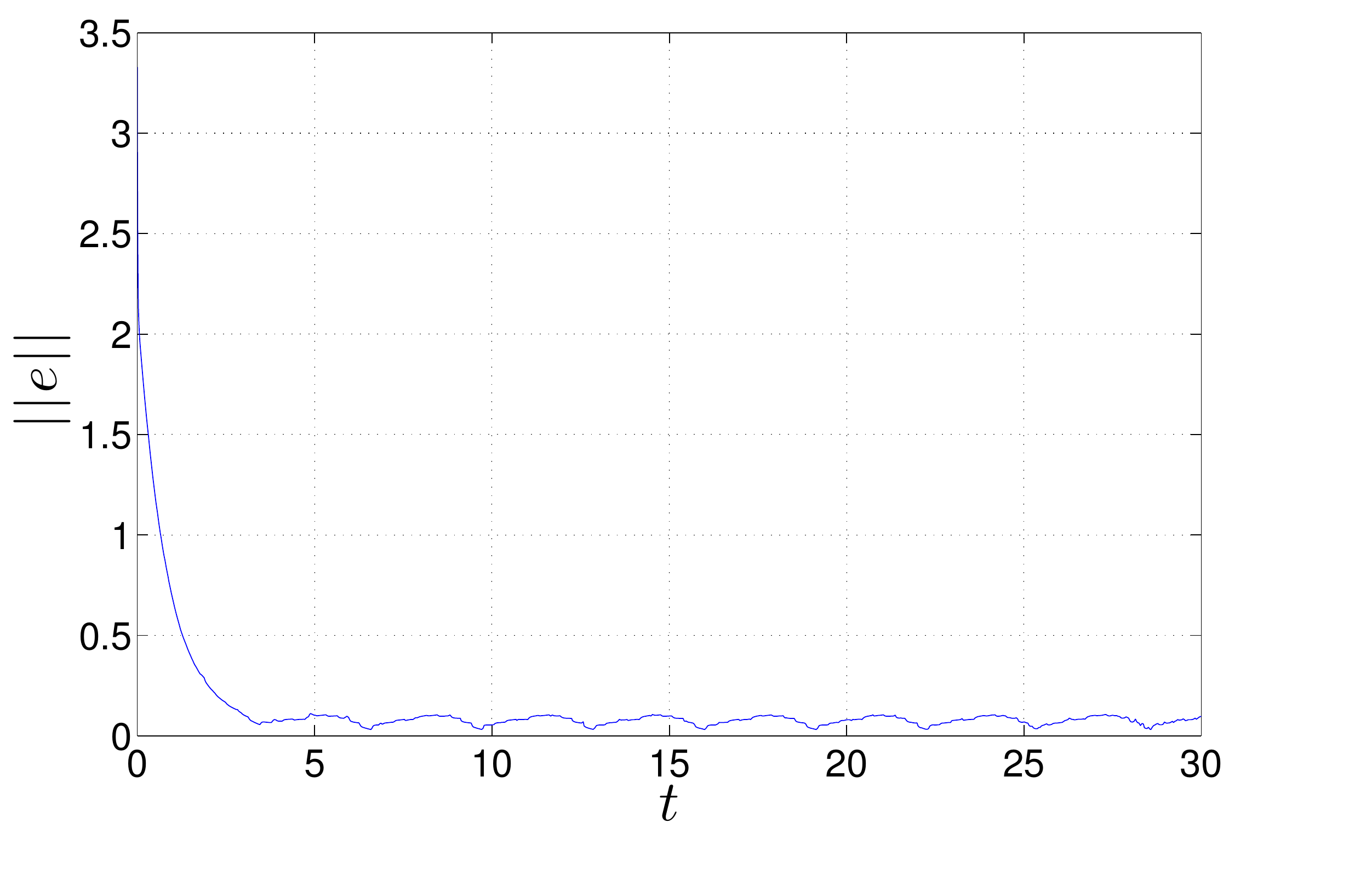}}
\caption{Time evolution of the norm of the synchronization error for a divergent Chua network: (a) uncoupled case; (b) coupled case.}
\label{fig:norma_errore_divergenti}
\end{figure}

\subsection{Networks of chaotic relays}\label{sec:relay}
Several examples of piecewise smooth systems whose dynamics are consistent with Assumption \ref{ass:partiallynonidentical} can be made. In particular, any QUAD system with a piecewise smooth feedback nonlinearity such as relay, saturation or hysteresis is also a QUAD Affine system. 
Here, we consider a network of five classical relay systems,  e.g. \cite{ts:85}, whose dynamics are described by:
\[
\dot{x}_i =Ax_i+B r_i, \quad y_i =C x_i, \quad r_i =-\mathrm{sgn}(y_i),
\]
where

{\singlespacing
\begin{equation}
A = \left[ {\begin{array}{*{20}{c}}
   {1.35} & 1 & 0  \\
   { - 99.93} & 0 & 1  \\
   { - 5} & 0 & 0  \\
\end{array}} \right] , \nonumber
\end{equation}}
\[
B=\left[1, -2, ~ 1\right]^T,
\quad
C=\left[1,~  0, ~ 0\right],
\]
%
As shown in \cite{dibu:07,dijo:01}, with this choice of parameter values,
 each relay exhibits both sliding motion and chaotic behavior.

The Laplacian matrix describing the network topology is

{\singlespacing
\begin{equation}
L = \left[ {\begin{array}{*{20}{c}}
     3  &  -1   &  0 &   -1 &   -1\\
    -1 &    4 &    -1   & -1  &  -1\\
     0 &   -1  &   2   & -1  &  0\\
    -1  &  -1  &  -1  &   4  &  -1\\
    -1   &  -1   &  0 &   -1   &  3\\
\end{array}} \right] , \nonumber
\end{equation}}

while the inner coupling matrix is $\Gamma=I_3$, that is, the nodes are coupled through all the state vector, and so the requirement $\Gamma\in\mathcal{D}^+$ of Corollary \ref{cor:partiallynonidentical} is satisfied. 

It easy to see that the network nodes satisfy Assumption \ref{ass:partiallynonidentical}. In particular, the QUAD term is $h(t,x_i)=Ax_i$ and the affine bounded (switching) term is $g_i(t,x_i)=Br_i$. Hence, we can use Corollary \ref{cor:partiallynonidentical} to obtain an upper bound on the minimum coupling gain guaranteeing $\epsilon$-bounded synchronization. Notice that, choosing for the sake of clarity $P=I_3$, we have
\begin{equation}
\begin{array}{l}
 {\left( {x - y} \right)^T}\left( {h(x) - h(y)} \right) = {\left( {x - y} \right)^T}A\left( {x - y} \right) =  \\ 
 {\left( {x - y} \right)^T}{A_{\mathrm{sym}}}\left( {x - y} \right) \le {\lambda _{\mathrm{max}}}\left( {{A_{\mathrm{sym}}}} \right){\left( {x - y} \right)^T}\left( {x - y} \right),\\ 
 \end{array}\nonumber
\end{equation}
where $A_{\mathrm{sym}}=\frac{1}{2}(A+A^T)$. 

In this example, we have $\lambda_{\mathrm{max}}(A_{\mathrm{sym}})=50$, while $\lambda_2(\Pi)=\lambda_2(L\otimes I_3)=2$. Therefore, with the choice of $P=I_3$ and from (\ref{eq:c_ub}), the lower bound $\tilde{c}$ is $25$.
In our simulation, we set the coupling gain $\hat{c}=50$, while the initial conditions are chosen randomly. Considering that $\xoverline[.75]{M}=2$, and using (\ref{eq:varepsilon_ub}), we can conclude that an upper bound for the norm of the stack error vector $e$  is $\bar{\epsilon}= 0.25$.
In Figures \ref{fig:rete5slidingcaotico_tx2}, \ref{fig:rete5connessa_te2(spessore)_ingrandimento}  and \ref{fig:rete5slidingcaotico}, we compare the behavior of the coupled network with the case of disconnected nodes. In particular, Figures \ref{fig:rete5slidingcaotico_tx2} and \ref{fig:rete5connessa_te2(spessore)_ingrandimento} show the time evolution of the second component of the synchronization error for each node, for both the uncoupled and coupled case, while Figure \ref{fig:rete5slidingcaotico} shows the evolution in the state space. 

Despite the presence of sliding motion, we observe the coupling to be effective in causing all nodes to synchronize, and the bound $\bar{\epsilon}\leq 0.25$ is consistent with what is observed in Figure \ref{fig:rete5slidingcaoticoconnessa_tx2}. 

\begin{figure}[h!]
\centering
\subfigure[]{\label{fig:rete5slidingcaoticononconnessa_tx2}\includegraphics[width=8.6cm]{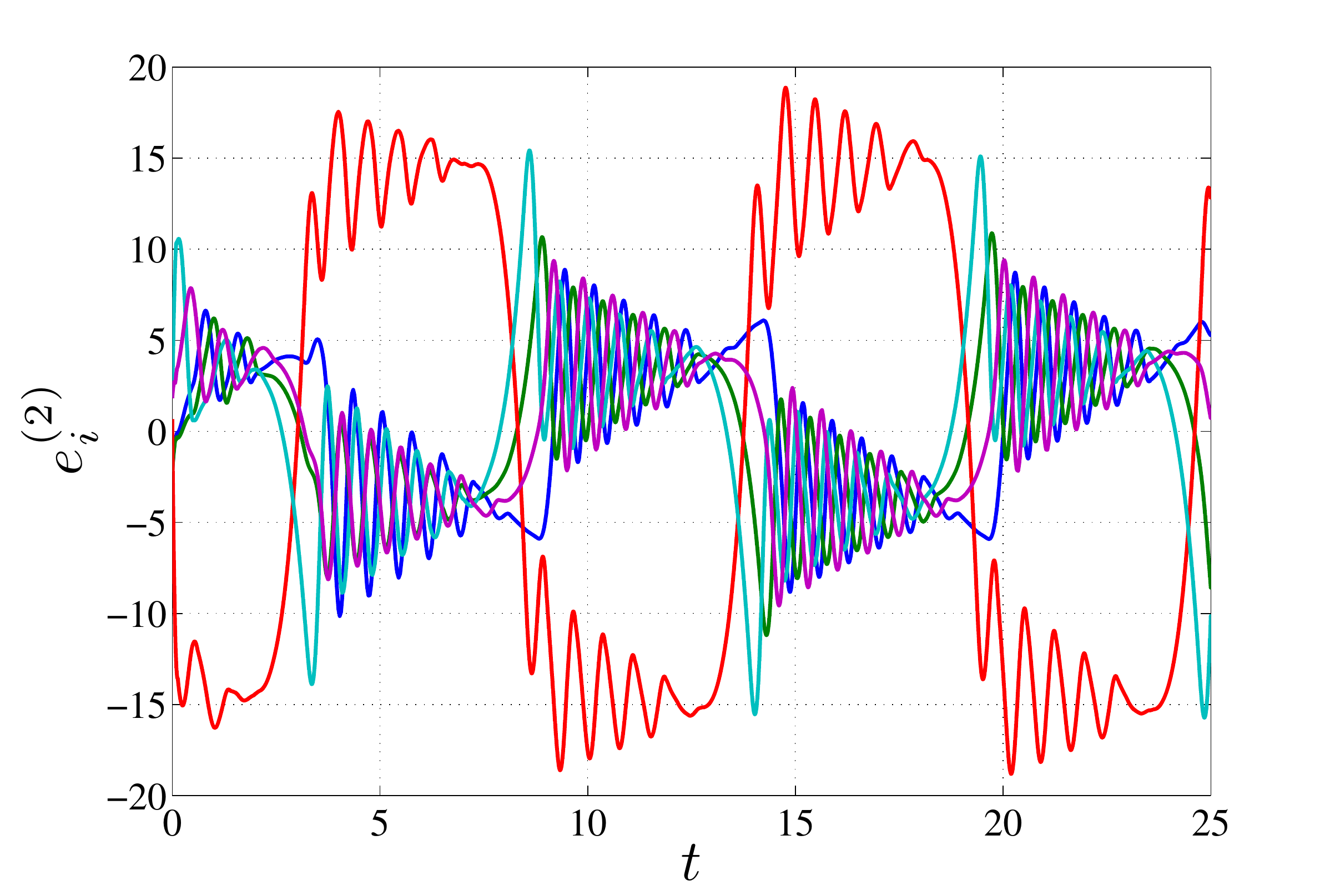}}
\subfigure[]{\label{fig:rete5slidingcaoticoconnessa_tx2}\includegraphics[width=8.6cm]{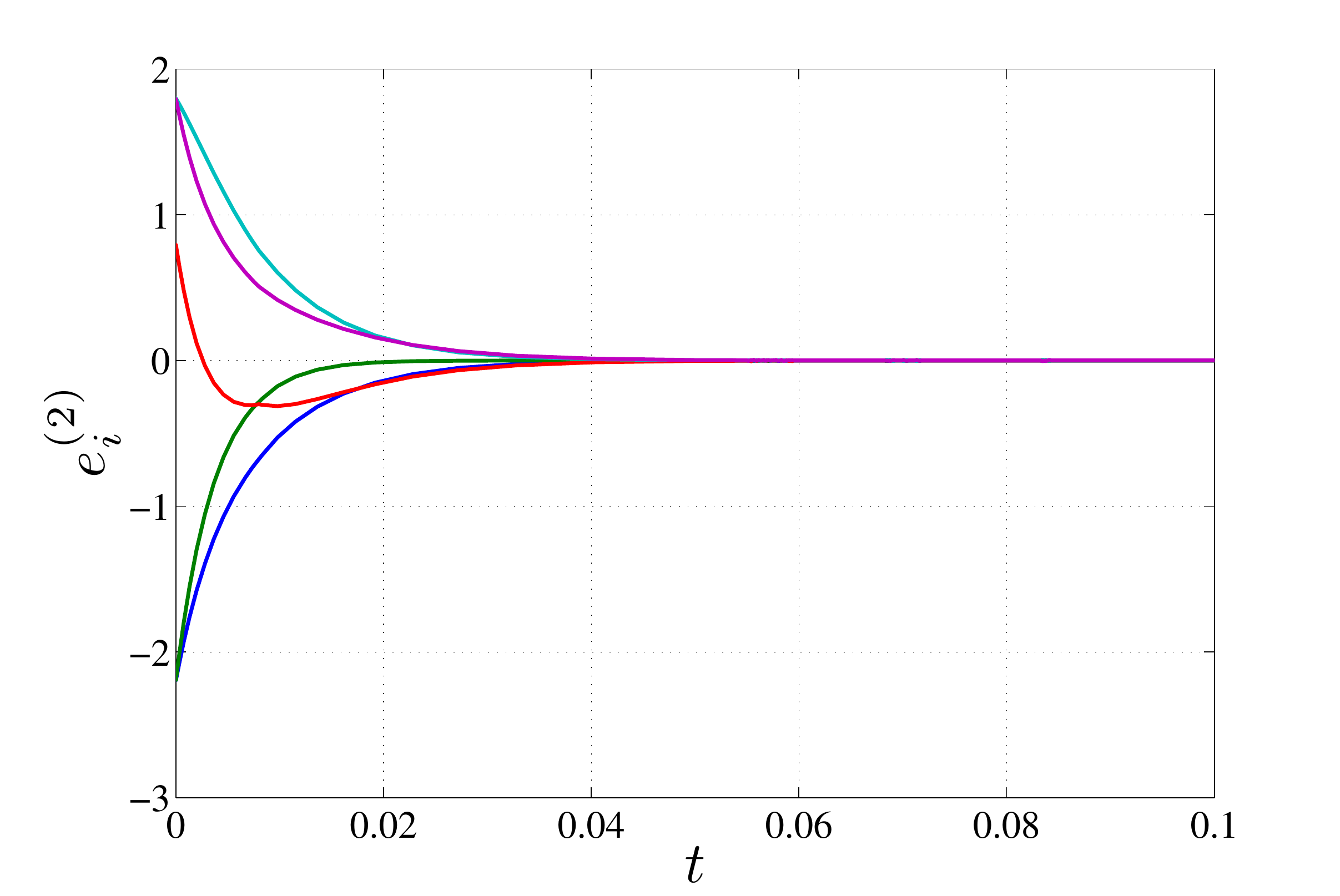}}
\caption{Time evolution of error components $e_{i}^{(2)}(t)$ for the network of chaotic relays: (a) uncoupled case; (b) coupled case.}
\label{fig:rete5slidingcaotico_tx2}
\end{figure}

\begin{figure}[h!]
  \centering 
  \includegraphics[width=0.7\columnwidth, keepaspectratio]{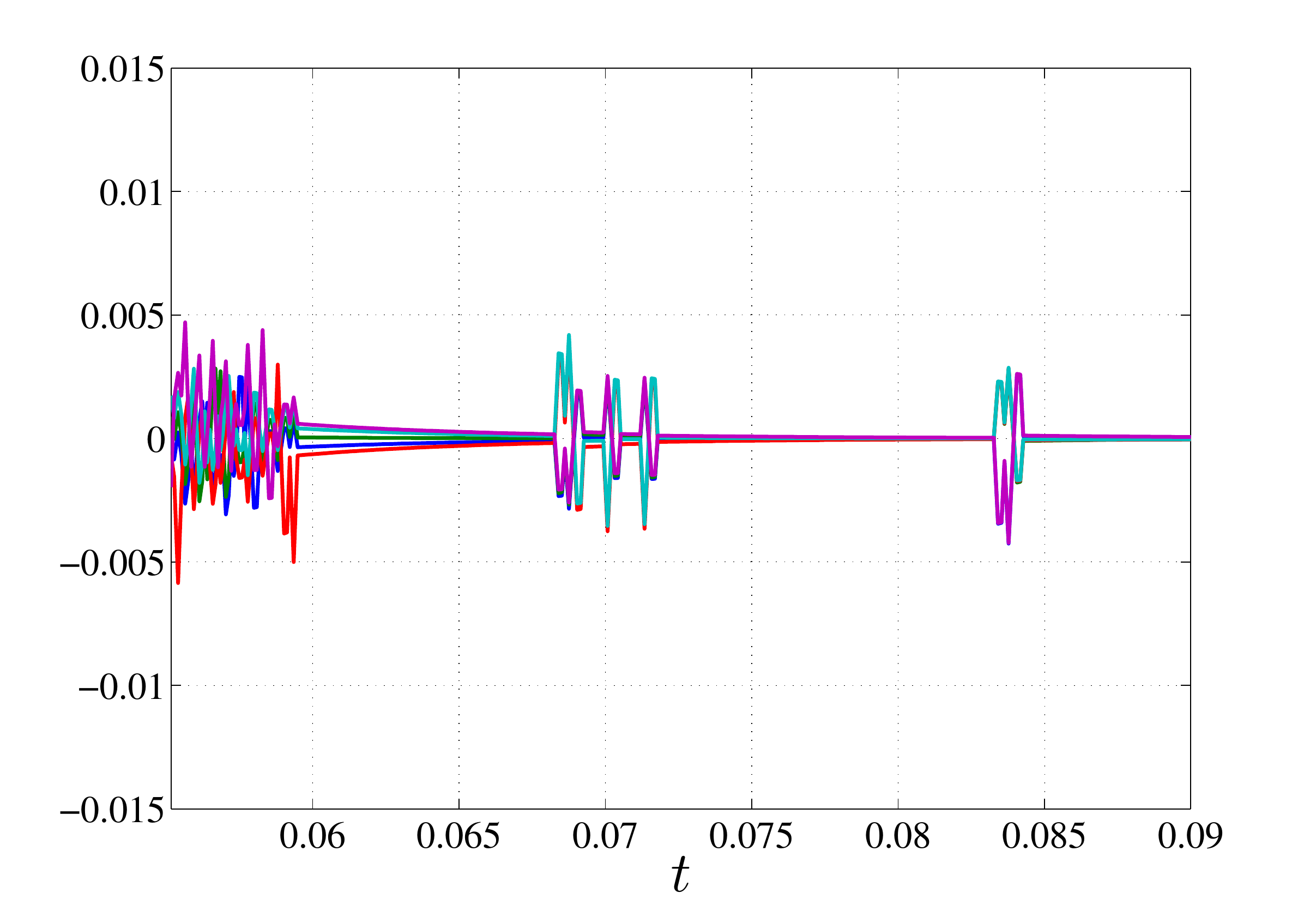}
  \caption{Zoom of the evolution of the error components $e_i^{(2)}(t)$ for $t>0.055s$ showing bounded convergence.}
  \label{fig:rete5connessa_te2(spessore)_ingrandimento}
\end{figure}
\begin{figure}[h!]
\centering
\subfigure[]{\label{fig:rete5nonconnessastatespace}\includegraphics[width=8.6cm]{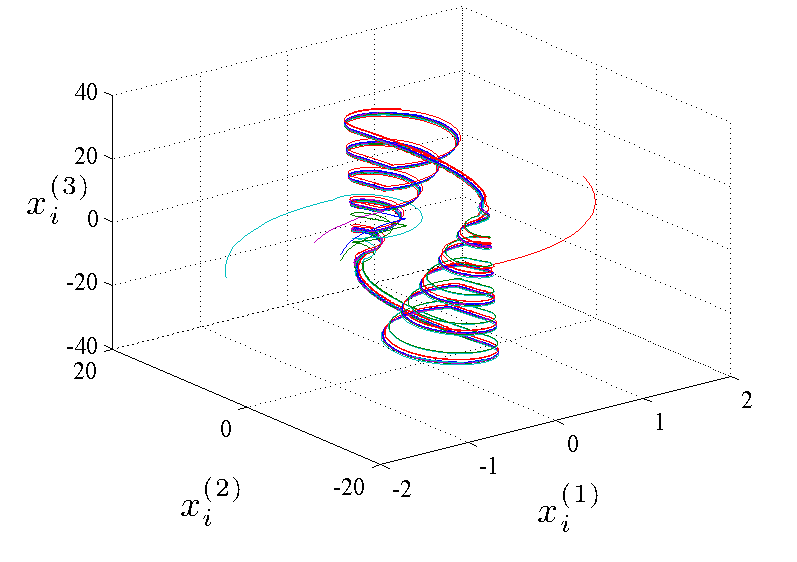}}
\subfigure[]{\label{fig:rete5connessastatespace}\includegraphics[width=8.6cm]{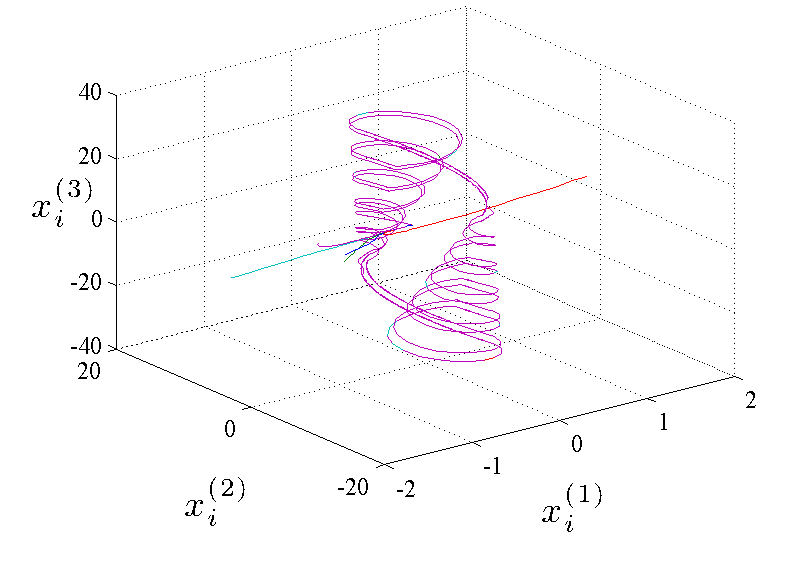}}
\caption{State space evolution for the network of chaotic relays: (a) uncoupled case; (b) coupled case.}
\label{fig:rete5slidingcaotico}
\end{figure}

\subsection{Nonuniform Kuramoto oscillators}\label{sec:kur}
A classical example of nonlinearly coupled heterogeneous systems is the network of nonuniform Kuramoto oscillators, described by equation
\begin{equation}\label{eq:nonuniform_kuramoto}
\dot{\theta}_i= \omega_i + \frac{K}{N}\sum_{j=1}^N a_{ij}\sin(\theta_j-\theta_i),\quad \theta_i\in(-\pi, \pi] \qquad i=1, \dots, N.
\end{equation}
Synchronization of Kuramoto oscillators has been widely  studied in literature, see for instance \cite{mist:05,st:00,camo:08,acbo:05}, where the coupling is generally supposed to be all-to-all, and ad hoc results about synchronization can be found. 

Here, we show how Theorem \ref{th:partiallinonidentical_nonlinear} can be also applied to a network of nonuniform Kuramoto systems \eqref{eq:nonuniform_kuramoto} and it provides an upper bound for the minimum coupling. The error system, defined as in equation \eqref{eq:error}, is given by
\begin{equation}\label{eq:nonuniform_kuramoto_error}
\dot{e}_i= \bar{\omega}_i + \frac{K}{N}\sum_{j=1}^N a_{ij}\sin(e_j-e_i),\quad \theta_i\in(-\pi, \pi] \qquad i=1, \dots, N.
\end{equation}
Now, if we take any initial condition $\theta(0)=[\theta_1(0),\dots,\theta_N(0)]^T$ such that $|\theta_i(0)-\theta_j(0)|<\pi$ for all $i,j=1,\ldots,N$, and if we set $c=K/N$, we have that each system in the network \eqref{eq:nonuniform_kuramoto_error} satisfies Assumption \ref{ass:partiallynonidentical} with $h=0$ and $g_i=\omega_i$.

In this example, we consider a network of $N=4$ nonuniform Kuramoto oscillators, whose topology is described by the Laplacian matrix

{\singlespacing
\begin{equation}
L = \left[ {\begin{array}{*{20}{rrrr}}
     2  &  1   &  0 &   1  \\
    -1 &    2 &    -1   & 0  \\
     0 &   -1  &   2   & -1  \\
    -1   &  0   &  -1 &   2  \\
\end{array}} \right] . \nonumber
\end{equation}}

The individual frequencies $\omega_i$ are taken from a normal distribution , while initial conditions are selected randomly in such that $|\theta_i(0)-\theta_j(0)|<e_{\max}=\pi/3$ for all $i,j=1,\ldots,N$. From Theorem \ref{th:partiallinonidentical_nonlinear}, we obtain an upper bound for the minimum coupling $\tilde{c}=0.73$. 
Figures \ref{fig:kuramoto_errore_disaccoppiati} and \ref{fig:kuramoto_errore_accoppiati} show the error trajectories for each oscillators, respectively for the case of uncoupled and coupled network with $K=3$ ($c=0.75$). 

\begin{figure}[h!]
\centering
\subfigure[]{\label{fig:kuramoto_errore_disaccoppiati}\includegraphics[width=8.6cm]{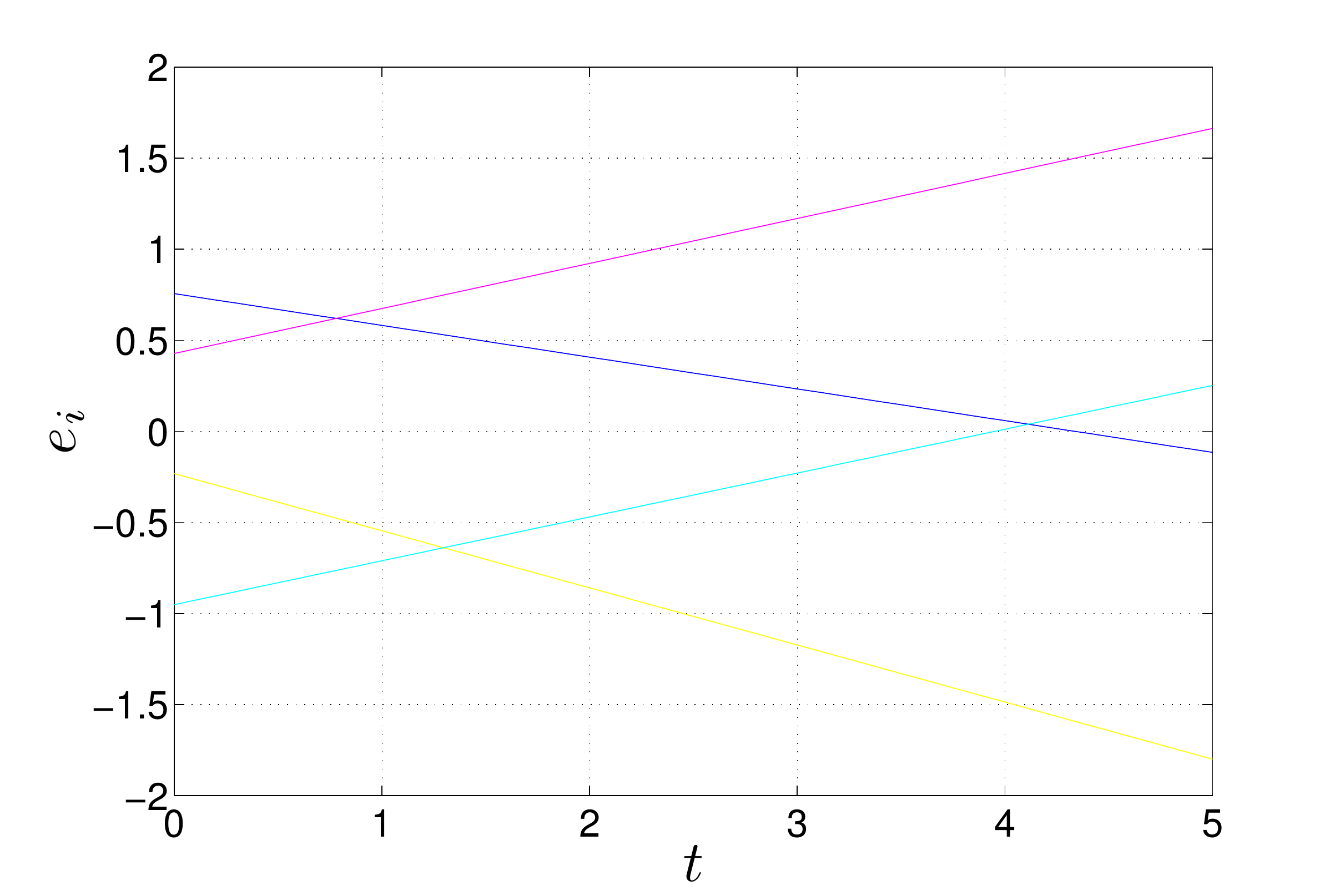}}
\subfigure[]{\label{fig:kuramoto_errore_accoppiati}\includegraphics[width=8.6cm]{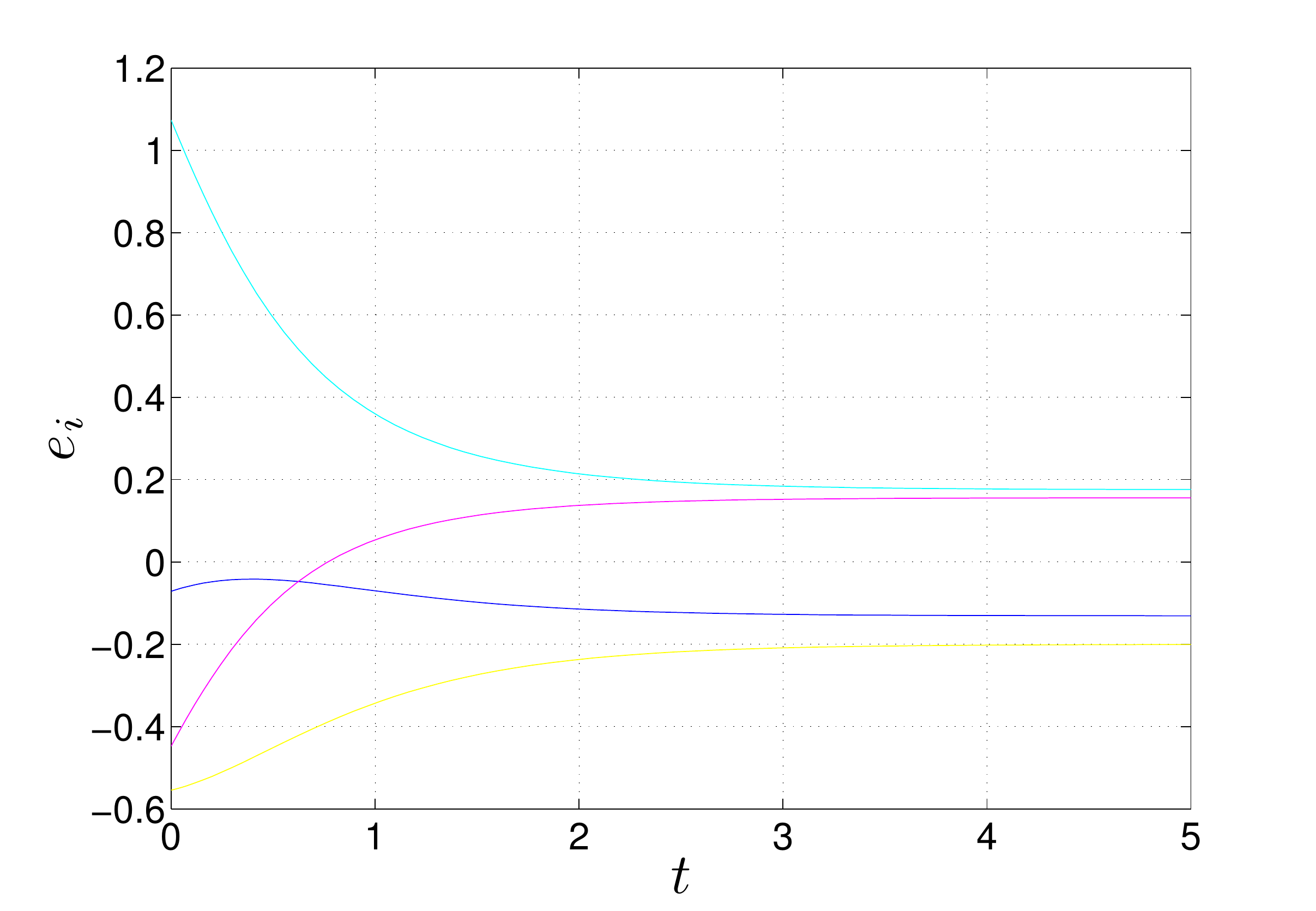}}
\caption{Time evolution of the errors for the network of Kuramoto oscillators: (a) uncoupled case; (b) coupled case.}
\label{fig:kuramoto_errore}
\end{figure}

For the uncoupled case the error diverges, while for the coupled case the global upper bound for the error norm predicted through Theorem \ref{th:partiallinonidentical_nonlinear} is $\bar{\epsilon}=1.04$, which is consistent with the value of $0.34$ that we obtain for the initial conditions given in our numerical simulation.

\section{Conclusions}\label{sec:con}
In this paper, we have presented a framework for the study of convergence and synchronization in networks whose nodes can be both piecewise smooth and/or nonidentical dynamical systems. Specifically, using a set-valued Lyapunov approach, we derived sufficient conditions for global convergence of all nodes towards the same bounded region of their state space and determined bounds on the minimum coupling strength required to achieve bounded synchronization. The residual synchronization bound $\epsilon$ was also estimated as a function of the mismatch between the nodes' dynamics and properties of both the network structure and the coupling function used across the network. 

Differently from previous approaches in the literature, we do not require that the trajectories of the coupled systems are bounded a priori or that conditions of synchrony among switching signals are satisfied. Also, the results presented in the paper allow to investigate convergence in networks of generic piecewise smooth systems including those exhibiting sliding motion, as the chaotic relay systems presented in Section \ref{sec:relay}. This represents a notable advantage of the approach presented in the paper when compared to what is currently available in the literature on networks of hybrid or piecewice smooth system.
The analysis has been performed both for linear and nonlinear coupling protocols and extensively validated on a set of representative numerical examples.

We wish to emphasise that the analytical tools presented in the paper can be used also to synthesise coupling functions and local controllers able to guarantee convergence of a network of interest. In particular, local nonlinear control actions can be added to nodes in a given network to make sure the relevant assumptions we use in our derivations are guaranteed together with appropriately designed coupling protocols. The in-depth investigation of our approach as a tool for the synthesis of distributed nonlinear control strategies for complex networked systems is currently under investigation and will be the subject of future of research.

\appendix
\section{Proof of Theorem \ref{th:partiallynonidentical_partiallycoupled}}\label{app:a}
Consider the following candidate Lyapunov function 
\begin{equation}
 V(e)=\frac{1}{2}e^T (I_N\otimes P) e,\label{eq:lyap}
\end{equation}
where we choose $P\in\mathcal Q^+$. Notice that $P\Gamma \ge 0$ and that the stack equation of the error evolution is given in \eqref{eq:error_stack_linear}. 
Evaluating the derivative of $V$ along the trajectory of such error system, and proceeding as in the step 2 of the proof
of Theorem \ref{th:totallynonidentical}, we get \eqref{eq:lin_tot_nonid_baru_l}. From Assumption \ref{ass:partiallynonidentical}, $h_i(t,\bar{x})=h_j(t,\bar{x})=h(t,\bar{x})$ for all $i,j=1,\ldots,N$. Now, from the properties of the Filippov set-valued function, and adding and subtracting $\sum_{i=1}^N e_i^TP{\mathop h\limits _{\sim}}(t,\bar{x})$, and using the product rule we can write $\mathcal{\overline{U}_L}\subseteq \mathcal{\overline{V}_L}$, where $\mathcal{\overline{V}_L}$ is 
\begin{align*}
\mathcal{\overline V_L}= 
\left\{\sum_{i=1}^N e_i^TP{\mathop h\limits _{\sim}}(t,x_i)+
e^T\left(I_N \otimes P\right)\mathop \Psi\limits _{\sim}+
\sum_{i=1}^N e_i^TP\mathop \xi\limits _{\sim}-
ce^T\left(L\otimes P\Gamma\right)e+
\sum_{i=1}^N e_i^TP{\mathop h\limits _{\sim}}(t,\bar{x})-
\sum_{i=1}^N e_i^TP{\mathop h\limits _{\sim}}(t,\bar{x}) \right\}
\end{align*}
with $\mathop \xi\limits _{\sim}\in\mathcal{F}\left[-\frac{1}{N}\sum_{j=1}^N f_j(t,x_j)\right]$.
As $\sum_{i=1}^N e_i=0$, we have that $\sum_{i=1}^N e_i^TP\mathop \xi\limits _{\sim}=0$ and $\sum_{i=1}^N e_i^TP{\mathop h\limits _{\sim}}(t,\bar{x})=0$. Considering the QUAD Affine assumption, and denoting $v_l$ a generic element of the set $\mathcal{V_L}$, the following inequality holds:
%
\[
\dot{V}(e)\leq v_l\le e^T\left[I_N\otimes W-cL \otimes P\Gamma\right]e+
e^T\left(I_N \otimes P\right)\mathop \Psi\limits ^{\sim}, \qquad \forall \mathop \Psi\limits ^{\sim}\in\mathop \Psi\limits _{\sim}\nonumber
\]
From trivial matrix properties, it follows that
\begin{equation}
\dot{V}(e)\leq e^T\left[I_N\otimes W-c L \otimes P\Gamma\right]e+\sqrt{N}\left\| e \right\|_{2} \left\| I_N\otimes P \right\|_{2} \xoverline[.75]{M},\label{eq:subset_vdot}
\end{equation}
with $\xoverline[.75]{M}=\max_{i=1,\dots, N} M_i$. 
Now, decomposing $e$ in $\tilde{e}_l$ and $\tilde e_{n-l}$ as in the proof of Theorem \ref{th:totallynonidentical}, we obtain
\begin{equation}
\dot{V}\le \left[ \lambda_{\mathrm{max}}(W_l)-c\lambda_2(L\otimes P_l\Gamma_l)\right]\tilde e_l^T \tilde e_l + \lambda_{\mathrm{max}}(W_{n-l})\tilde e_{n-l}^T\tilde e_{n-l}^T +\sqrt{N}||e||_2||P||_2\xoverline[.75]{M} .
\nonumber
\label{eq:subset_vdot2}
\end{equation}
Defining $m(c,P,W)$ according to \eqref{eq:mcp_partiallynonidentical}, and rewriting the synchronization error as $e=a\hat e$, with $\hat e=\frac{e}{\left\| e \right\|_{2}}$, we obtain
\begin{align}
\dot{V} \le &-m(c,P,W)e^T e+\sqrt{N}||e||_2||P||_2\xoverline[.75]{M}\nonumber\\
			  = &-m(c,P,W)a^2+a \xoverline[.75]{M} \sqrt{N}||P||_2.
\label{eq:subset_vdot_final}
\end{align}
If we choose any $c>\tilde{c}$, then we can always select a couple $(P,W)\in\mathcal{PW}$ such that $m(c,P,W)>0$. Then, from \eqref{eq:subset_vdot_final}, we have that $a>\xoverline[.75]{M}\sqrt{N}||P||_2/m(c,P,M)$ implies $\dot{V}<0$.
Hence, network \eqref{eq:genericnetwork_linear}  is $\epsilon$-bounded synchronized with
\begin{equation}
\epsilon \le \frac{\xoverline[.75]{M}\sqrt{N}||P||_2}{m(c,P,M)}.
\label{eq:subset_epsilon}
\end{equation}
Bound \eqref{eq:varepsilon_ub_sub} follows trivially from \eqref{eq:subset_epsilon}.

\section{Proof of Theorem \ref{th:partiallinonidentical_nonlinear}}\label{app:b}
Considering the candidate Lyapunov function $V(e)=\frac{1}{2}e^Te$, and evaluating its derivative as in the step 2 of the proof of Theorem \ref{th:totallynonidentical_nonlinear}, we obtain equation \eqref{eq:oveline_U}. Adding and subtracting $\sum_{i=1}^N e_i^T{\mathop h\limits _{\sim}}(t,\bar{x})$, and using the product rule, we have that
\begin{align*}
\dot{V}(e)\in\mathcal{\overline V_L}= &
\left\{\sum_{i=1}^N e_i^T{\mathop h\limits _{\sim}}(t,x_i)+
e^T\mathop \Psi\limits _{\sim}+
\sum_{i=1}^N e_i^T\mathop \xi\limits _{\sim}-\frac{1}{2}
c\sum_{i=1}^N\sum_{j=i}w_{ij}(e_i-e_j)^T{\mathop \eta\limits _{\sim}}(t,e_i-e_j)
+\sum_{i=1}^N e_i^T{\mathop h\limits _{\sim}}(t,\bar{x}) \right.\\
& \left.
- \sum_{i=1}^N e_i^T{\mathop h\limits _{\sim}}(t,\bar{x})\right\},
\end{align*}
with $\mathop \xi\limits _{\sim}\in\mathcal{F}\left[-\frac{1}{N}\sum_{j=1}^N f_j(t,x_j)\right]$.
As $\sum_{i=1}^N e_i=0$, we have that $\sum_{i=1}^N e_i^T\mathop \xi\limits _{\sim}=0$, and $\sum_{i=1}^N e_i^T{\mathop h\limits _{\sim}}(t,\bar{x})=0$. From Assumptions \ref{ass:partiallynonidentical} and \ref{ass3}, the following inequality holds:
\begin{equation}
\dot{V}(e)\leq e^T\left[I_N\otimes W-cL \otimes \Upsilon\right]e+
e^T\mathop \Psi\limits ^{\sim}, \qquad \forall \mathop \Psi\limits ^{\sim}\in\mathop \Psi\limits _{\sim}.\nonumber
\end{equation}
Notice that, as in the proof of Theorem \ref{th:totallynonidentical_nonlinear}, the upper bound $\tilde{c}$ for the minimum coupling and hypothesis (ii) guarantee that the inequality \eqref{eq:nonlinear_coupling_inequality} is always feasible.
Decomposing the error vector $e$ in the two parts $\tilde{e}_r$ and $\tilde{e}_{n-r}$ and following similar steps as in Theorem \ref{th:partiallynonidentical_partiallycoupled}, the thesis follows. 
%
%

\end{document}